\documentclass[11pt]{amsart}
\usepackage{amsmath}
\usepackage{amssymb}
\usepackage{amsthm}
\usepackage{amscd}
\usepackage{enumerate}
\usepackage{verbatim}
\usepackage[all]{xy}
\usepackage{mathrsfs}

\theoremstyle{plain}
\newtheorem{thm}{Theorem}[section]
\newtheorem{prop}[thm]{Proposition}
\newtheorem{lem}[thm]{Lemma}
\newtheorem{cor}[thm]{Corollary}

\theoremstyle{definition}
\newtheorem{dfn}[thm]{Definition}
\newtheorem{rmk}[thm]{Remark}

\newcommand{\rank}{\mathrm{rank}}

\newcommand{\Hom}{\mathrm{Hom}}

\newcommand{\Ker}{\mathrm{Ker}}

\newcommand{\prjt}{\mathrm{pr}}

\newcommand{\Spa}{\mathrm{Spa}}

\newcommand{\Spv}{\mathrm{Sp}}
\newcommand{\Spec}{\mathrm{Spec}}

\newcommand{\Gal}{\mathrm{Gal}}
\newcommand{\ad}{\mathrm{ad}}
\newcommand{\Frac}{\mathrm{Frac}}
\newcommand{\id}{\mathrm{id}}

\newcommand{\ch}{\mathrm{char}}
\newcommand{\Map}{\mathrm{Map}}

\newcommand{\geom}{\mathrm{geom}}

\newcommand{\sep}{\mathrm{sep}}

\newcommand{\ra}{\rangle}
\newcommand{\la}{\langle}

\newcommand{\Kbar}{\bar{K}}
\newcommand{\kbar}{\bar{k}}

\newcommand{\okbar}{\mathcal{O}_{\bar{K}}}

\newcommand{\fR}{\mathbb{C}^\flat}
\newcommand{\oR}{\mathcal{O}_{\mathbb{C}^{\flat}}}

\newcommand{\oee}{\mathcal{O}_E}
\newcommand{\okey}{\mathcal{O}_K}
\newcommand{\oc}{\mathcal{O}_{\mathbb{C}}}
\newcommand{\oel}{\mathcal{O}_L}

\newcommand{\oem}{\mathcal{O}_M}

\newcommand{\cF}{\mathcal{F}}

\newcommand{\cO}{\mathcal{O}}

\newcommand{\Ksep}{K^{\mathrm{sep}}}

\newcommand{\upi}{\underline{\pi}}
\newcommand{\ux}{\underline{x}}

\newcommand{\uy}{\underline{y}}

\newcommand{\oef}{\mathcal{O}_F}

\newcommand{\bA}{\mathbb{A}}

\newcommand{\bC}{\mathbb{C}}

\newcommand{\bZ}{\mathbb{Z}}
\newcommand{\bQ}{\mathbb{Q}}
\newcommand{\bF}{\mathbb{F}}

\newcommand{\sB}{\mathscr{B}}

\newcommand{\frp}{\mathfrak{p}}
\newcommand{\frq}{\mathfrak{q}}

\makeatletter
\renewcommand{\p@enumii}{}
\makeatother

\begin{document}

\title[Ramification theory and perfectoid spaces]{Ramification theory and perfectoid spaces}
\author{Shin Hattori}
\date{\today}
\email{shin-h@math.kyushu-u.ac.jp}
\address{Faculty of Mathematics, Kyushu University}

\begin{abstract}
Let $K_1$ and $K_2$ be complete discrete valuation fields of residue characteristic $p>0$. Let $\pi_{K_1}$ and $\pi_{K_2}$ be their uniformizers. Let $L_1/K_1$ and $L_2/K_2$ be finite extensions with compatible isomorphisms of rings $\cO_{K_1}/(\pi_{K_1}^m)\simeq \cO_{K_2}/(\pi_{K_2}^m)$ and $\cO_{L_1}/(\pi_{K_1}^m)\simeq \cO_{L_2}/(\pi_{K_2}^m)$ for some positive integer $m$ which is no more than the absolute ramification indices of $K_1$ and $K_2$. Let $j\leq m$ be a positive rational number. In this paper, we prove that the ramification of $L_1/K_1$ is bounded by $j$ if and only if the ramification of $L_2/K_2$ is bounded by $j$. As an application, we prove that the categories of finite separable extensions of $K_1$ and $K_2$ whose ramifications are bounded by $j$ are equivalent to each other, which generalizes a theorem of Deligne to the case of imperfect residue fields. We also show the compatibility of Scholl's theory of higher fields of norms with the ramification theory of Abbes-Saito, and the integrality of small Artin and Swan conductors of abelian extensions of mixed characteristic. 
\end{abstract}

\maketitle
\tableofcontents
\section{Introduction}\label{intro}

Let $K$ be a complete discrete valuation field of residue characteristic $p>0$. Let $k$ be the residue field, $\okey$ be the ring of integers and $\pi=\pi_K$ be a uniformizer of $K$. We define $e(K)$ to be the absolute ramification index of $K$ if $\ch(K)=0$ and an arbitrary positive integer if $\ch(K)=p$.

When $k$ is perfect, the classical ramification theory defines a notion of ramification of any finite separable extension $L/K$ and, for any positive rational number $j$, a notion of whether the ramification of $L/K$ is bounded by $j$ (see \cite{Se}). We let $\mathrm{FE}^{\leqslant j}_K$ denote the category of finite separable extensions $L/K$ whose ramification is bounded by $j$.

On the other hand, Deligne (\cite{De}) defined a ramification theory of truncated discrete valuation rings with perfect residue fields. Let $m$ be a positive integer. A truncated discrete valuation ring of length $m$ is by definition a local ring $A$ with principal maximal ideal which is nilpotent such that $A$ is of length $m$ as an $A$-module. The ring $\okey/(\pi^m)$ is a truncated discrete valuation ring of length $m$, and conversely any truncated discrete valuation ring of length $m$ can be written as such a quotient of the ring of integers of some complete discrete valuation field. For the case where the residue field of $A$ is perfect, he defined a notion of finite extension $B/A$ of truncated discrete valuation rings and a notion of whether its ramification is bounded by $j$, for any positive rational number $j$ satisfying $j\leq m$. Moreover, for any truncated discrete valuation ring $A$ of length $m$ with perfect residue field and $j\leq m$, he also defined a category $(\mathrm{ext}\, A)^j$ of finite extensions $B/A$ whose ramification is bounded by $j$.

Depending on the choice of a presentation $A\simeq \okey/(\pi^m)$ of $A$ as a quotient of a complete discrete valuation ring $\okey$, we have a natural functor $\mathrm{FE}^{\leqslant j}_K\to (\mathrm{ext}\, A)^j$ defined by $L\mapsto \oel/(\pi^m)$. Then Deligne also showed that this is an equivalence of categories. A striking fact is that the category $(\mathrm{ext}\, A)^j$ is independent of the choice of a presentation $A\simeq \okey/(\pi^m)$. This implies that, for any complete discrete valuation fields $K_1$ and $K_2$ with perfect residue fields of characteristic $p$, if there exists a ring isomorphism $\cO_{K_1}/(\pi_{K_1}^m)\simeq \cO_{K_2}/(\pi_{K_2}^m)$, then the categories $\mathrm{FE}^{\leqslant j}_{K_1}$ and $\mathrm{FE}^{\leqslant j}_{K_2}$ are equivalent, even though the characteristics of $K_1$ and $K_2$ may be different. A key point of this equivalence is that, since the residue field $k$ is assumed to be perfect, for any finite separable extension $L/K$ the $\okey$-algebra $\oel$ is generated by a single element $x$, and ramification of the extension $L/K$ can be read off from the Newton polygon of (a translation of) the minimal polynomial of $x$, which is a combinatorial object independent of $\ch(K)$.

For the case where the residue field $k$ is imperfect, a ramification theory of finite separable extensions of $K$ was developed satisfactorily by Abbes and Saito (\cite{AS1} and \cite{AS2}), and we have a category $\mathrm{FE}^{\leqslant j}_K$ of finite separable extensions $L/K$ whose (non-log) ramification is bounded by $j$, as in the case of perfect residue field. In their ramification theory, the notion of whether the (non-log and log) ramification of a finite separable extension $L/K$ is bounded by some positive rational number $j$ is defined by counting the number of geometric connected components of a tubular neighborhood of defining equations of the $\okey$-algebra $\oel$ in the sense of rigid analytic geometry. Note that, in this case, the $\okey$-algebra $\oel$ is not necessarily generated by a single element and thus it seems difficult to control its ramification by Newton polygons. 

Using their works and the author's (\cite{Ha}), Hiranouchi and Taguchi (\cite{HT}) defined, for any truncated discrete valuation ring $A$ of length $m$ whose residue field may be imperfect and any positive rational number $j\leq m$, a category $\mathrm{FFP}^{\leqslant j}_A$ of finite extensions $B/A$ whose ramification is bounded by $j$ (see Definition \ref{defFFP}). In fact, they defined the category by choosing a presentation of $A$ as above. They questioned whether it is independent of the choice, and whether we can generalize the striking equivalence of Deligne to the case of imperfect residue field. 

In this paper, we prove the following correspondence result of (non-log and log) ramification of finite extensions of complete discrete valuation fields which may have different characteristics.

\begin{thm}\label{maincomparison}
Let $L_1/K_1$ and $L_2/K_2$ be finite extensions of complete discrete valuation fields of residue characteristic $p>0$. Let $\pi_{K_i}$ be a uniformizer of $K_i$. Let $m$ be a positive integer satisfying $m \leq \min_i e(K_i)$. Suppose that we have compatible isomorphisms of rings $\cO_{K_1}/(\pi_{K_1}^m)\simeq \cO_{K_2}/(\pi_{K_2}^m)$ and $\cO_{L_1}/(\pi_{K_1}^m)\simeq \cO_{L_2}/(\pi_{K_2}^m)$. 
\begin{enumerate}
\item{(Corollary \ref{LEnonlog})} For any positive rational number $j\leq m$, the ramification of $L_1/K_1$ is bounded by $j$ if and only if the ramification of $L_2/K_2$ is bounded by $j$.
\item{(Corollary \ref{LElog})} For any positive rational number $j\leq m-2$, the log ramification of $L_1/K_1$ is bounded by $j$ if and only if the log ramification of $L_2/K_2$ is bounded by $j$.
\end{enumerate}
\end{thm}

Note that a similar correspondence of ramification is studied by the author for the case of finite flat group schemes (\cite{Ha_ramcorr} and \cite{Ha_lowram}). As an application of Theorem \ref{maincomparison}, we answer the above questions of Hiranouchi-Taguchi affirmatively for the case of $pA=0$, as follows. 

\begin{thm}\label{mainHT}
\begin{enumerate}
\item{(Theorem \ref{CatIndep})} The category $\mathrm{FFP}^{\leqslant j}_A$ is independent of the choice of a presentation $A\simeq \okey/(\pi^m)$.
\item{(Corollary \ref{equivcats})}
Let $K_1$ and $K_2$ be complete discrete valuation fields with residue fields $k_1$ and $k_2$ of characteristic $p>0$, respectively. Let $j$ be a positive rational number satisfying $j\leq \min_i e(K_i)$. Suppose that the fields $k_1$ and $k_2$ are isomorphic to each other. Then there exists an equivalence of categories
\[
\mathrm{FE}_{K_1}^{\leqslant j} \simeq \mathrm{FE}_{K_2}^{\leqslant j}.
\]
In particular, there exists an isomorphism of topological groups
\[
G_{K_1}/G_{K_1}^j\simeq G_{K_2}/G_{K_2}^j,
\]
where $G_{K_i}^j$ is the $j$-th (non-log) upper ramification subgroup of the absolute Galois group $G_{K_i}$ (\cite[Section 3]{AS1}).
\end{enumerate}
\end{thm}

We also give the following applications of Theorem \ref{maincomparison} to Scholl's theory of higher fields of norms (\cite{Scholl}) and the integrality of conductors of an abelian extension of $K$. Note that Theorem \ref{mainscholl} was proved by Shun Ohkubo using a totally different method.

\begin{thm}{(Theorem \ref{Schollthm})}\label{mainscholl}
The functor of higher fields of norms is compatible with (non-log and log) ramification.
\end{thm}

\begin{thm}{(Theorem \ref{HA})}
Suppose $\ch(K)=0$. Let $L/K$ be a finite abelian extension. Let $c(L/K)$ ({\it resp.} $c_{\log}(L/K)$) be the Artin conductor ({\it resp.} Swan conductor) of the extension $L/K$.
\begin{enumerate}
\item\label{HAnonlogmain}
If $c(L/K)< e(K)$, then $c(L/K)$ is an integer.
\item\label{HAlogmain}
If $c_{\log}(L/K)< e(K)-2$, then $c_{\log}(L/K)$ is an integer.
\end{enumerate}
\end{thm}

The key idea of the proof of Theorem \ref{maincomparison} is to compare the sets of geometric connected components of affinoid varieties of different characteristics using the theory of perfectoid spaces due to Scholze (\cite{Sch}). By a base change, we reduce ourselves to such a comparison of the case where the residue field $k$ is perfect. Namely, we consider the following situation: we have a diagram of surjections
\[
k[[u]]\to A \gets \okey,
\]
where the images of $\pi$ and $u$ in $A$ coincide, and we also have a set of polynomials $\bar{f}=\{\bar{f}_1,\ldots,\bar{f}_r\}$ in $A[X]$. Here we put $X=(X_1,\ldots,X_n)$. Let $f\subseteq \okey[X]$ and $\mathbf{f}\subseteq k[[u]][X]$ be lifts of $\bar{f}$. Let $\bC$ be the completion of an algebraic closure of $K$. Let $\fR$ be its tilt (\cite[Section 3]{Sch}), which is defined as the fraction field of the inverse limit ring
\[
\oR=\varprojlim_\Phi \oc/(\pi^m)
\]
along the $p$-th power Frobenius map. The field $k((u))$ is considered as a subfield of $\fR$ by $u\mapsto \upi$, where we define $\upi=(\pi_l)_{l\geq 0}$ by choosing a system of $p$-power roots of $\pi$ in $\bC$ satisfying $\pi_0=\pi$ and $\pi_{l+1}^p=\pi_l$. Consider the adic spaces over $\bC$
\[
X_{\bC,0}^\ad=\Spa(\bC\la X\ra,\oc\la X\ra),\ X_{\bC,\infty}^\ad=\Spa(\bC\la X^{1/p^\infty}\ra,\oc\la X^{1/p^\infty}\ra)
\]
and also similar adic spaces $X_{\fR,0}^\ad$ and $X_{\fR,\infty}^\ad$ over $\fR$. Then we have a diagram
\[
\xymatrix{
X_{\bC,\infty}^\ad\ar[r]^-\tau_-{\sim}\ar[d] & X_{\fR,\infty}^\ad\ar[d]\\
X_{\bC,0}^\ad & X_{\fR,0}^\ad,
}
\]
where the map $\tau$ is the homeomorphism of \cite[Theorem 6.3]{Sch}. The equations $f$ and $\mathbf{f}$ define the rational subsets $X^{j,\ad}_{\bC}\subseteq X^\ad_{\bC,0}$ and $X^{j,\ad}_{\fR}\subseteq X^\ad_{\fR,0}$ given by the inequalities
\[
|f_i(x)|\leq |\pi(x)|^j \text{ and } |\mathbf{f}_i(x)|\leq |u(x)|^j,
\]
respectively. Here $|\cdot (x)|$ denotes the associated continuous valuation for any point $x$ of these adic spaces. The inverse image of $X^{j,\ad}_{\fR}$ in $X_{\bC,\infty}^\ad$ by the composite in the above diagram is the rational subset given by the inequality
\[
|\mathbf{f}_i^\sharp(x)|\leq |\pi(x)|^j,
\]
where $(\cdot)^\sharp:\oR\la X^{1/p^\infty} \ra\to \oc \la X^{1/p^\infty} \ra$ is a natural multiplicative map (\cite[Theorem 6.3]{Sch}). From the choice of $f$ and $\mathbf{f}$, we can prove the congruence
\[
\mathbf{f}_i^\sharp\equiv f_i\bmod \pi^m.
\]
Thus the assumption on $j$ implies that the inverse image coincides with the inverse image of $X^{j,\ad}_{\bC}$ in $X_{\bC,\infty}^\ad$. Then Theorem \ref{maincomparison} follows by showing that the vertical arrows of the above diagram induce bijections between the sets of connected components of the rational subsets $X^{j,\ad}_{\bC}$, $X^{j,\ad}_{\fR}$ and their inverse images in $X^\ad_{\bC,\infty}$, $X^{\ad}_{\fR,\infty}$.

\noindent
{\bf Acknowledgments.} The author would like to thank Yuichiro Taguchi for stimulating discussions, answering questions on the paper \cite{HT}, valuable comments on earlier drafts and the invitation to Korea Institute for Advanced Study (KIAS), where a part of this work was carried out. He is grateful for the hospitality provided by KIAS. He also would like to thank Shun Ohkubo for pointing out that Theorem \ref{mainscholl} would follow if we could generalize the aforementioned theorem of Deligne to the case of imperfect residue field, and Yoichi Mieda for helpful comments on adic spaces which improved proofs. This work was supported by JSPS KAKENHI Grant Number B-23740025.



\section{Lemmas on connected components of analytic spaces}

Let $K$ be a complete valuation field of rank one. Let $\Ksep$ be a separable closure of $K$, which we consider as a valuation field by extending the valuation of $K$ naturally. Let $\bC$ be the completion of $\Ksep$. In this section, we show lemmas which compare the sets of connected components in various settings of analytic geometry over $K$. First we show the following lemma comparing the sets of connected components between $K$-affinoid varieties and their associated adic spaces (\cite{Hu_Fo}).

\begin{lem}\label{adrigconn}
Let $A$ be a $K$-affinoid algebra in the sense of \cite[Definition 6.1.1/1]{BGR} and $\mathring{A}$ be the subring of power-bounded elements of $A$. Let $X=\Spv(A)$ be its associated $K$-affinoid variety and $X^\ad=\Spa(A,A^\circ)$ be its associated adic space. Then we have a natural bijection $\pi_0(X^\ad) \to \pi_0(X)$.
\end{lem}
\begin{proof}
The set $X$ is naturally considered as a subset of $X^\ad$. Since $X$ is quasi-separated, the association $U \mapsto U\cap X$ gives a bijection from the set of quasi-compact open subsets of $X^\ad$ to the set of quasi-compact admissible open subsets of $X$. Moreover, the notions of open covering and admissible open covering correspond to each other by this bijection (\cite[(1.1.11)]{Hu_Et}). 

Let $U$ be a quasi-compact admissible open subset of $X$ and $U^\ad$ be the associated quasi-compact open subset of $X^\ad$ via the above bijection. We first prove that if $U$ is connected, then $U^\ad$ is also connected. Indeed, suppose that we have a decomposition $U^\ad=V_1^\ad \coprod V_2^\ad$ of $U^\ad$ into the disjoint union of open subsets $V_i^\ad$. Since $U^\ad$ is quasi-compact, the open subsets $V_i^\ad$ are also quasi-compact. Put $V_i=V_i^\ad\cap X$, which is a quasi-compact admissible open subset of $X$. We have $U=V_1\coprod V_2$ and this is an admissible open covering by the above bijection. Since $U$ is connected, we may assume $U=V_1$ and thus we obtain $U^\ad=V_1^\ad$. This implies that $U^\ad$ is connected.

Let $U_1,\ldots, U_n$ be the connected components of $X$. Each $U_i$ is a rational subdomain of $X$ and thus quasi-compact. Let $U_i^\ad$ be the associated quasi-compact open subset of $X^\ad$. Since the covering $X=\coprod_{i=1}^n U_i$ is an admissible open covering, the above bijection shows $X^\ad=\cup_{i=1}^n U_i^\ad$. Since $U_i^\ad\cap U_j^\ad$ is constructible and $X\cap(U_i^\ad\cap U_j^\ad)=U_i\cap U_j$, \cite[Corollary 4.2]{Hu_C} implies $U_i^\ad\cap U_j^\ad=\emptyset$ for any $i\neq j$. Thus each $U_i^\ad$ is a connected component of $X^\ad$ and the lemma follows. 
\end{proof}

Let $A$ be a $K$-affinoid algebra which is geometrically reduced. We define the set of geometric connected components of $\Spv(A)$ as
\[
\pi_0^\geom(\Spv(A))=\varprojlim_{L/K} \pi_0(\Spv(A\otimes_K L)),
\]
where the limit on the right-hand side runs over the category of finite separable extensions of $K$ in $\Ksep$. This set is a finite set and the inverse system is constant for any sufficiently large $L$, by the reduced fiber theorem (\cite[Theorem 1.3]{BLR4}. See also \cite[Theorem 4.2]{AS1}).  It has a natural continuous action of the absolute Galois group $G_K=\Gal(\Ksep/K)$.

\begin{lem}\label{connCp}
Let $A$ be a geometrically reduced $K$-affinoid algebra. Then there exists a natural isomorphism of finite $G_K$-sets
\[
\pi_0(\Spa(A\hat{\otimes}_K \bC, (A\hat{\otimes}_K \bC)^\circ))\to \pi_0^\geom(\Spv(A)).
\]
\end{lem}
\begin{proof}
For any extension $L/K$ of complete valuation fields of rank one, the ring $A\hat{\otimes}_K L$ is an $L$-affinoid algebra and we put
\[
X^\ad_L=\Spa(A\hat{\otimes}_K L,(A\hat{\otimes}_K L)^\circ).
\]
By Lemma \ref{adrigconn}, it suffices to show that the natural map of finite $G_K$-sets
\[
\pi_0(X^\ad_\bC)\to \varprojlim_{L/K}\pi_0(X_L^\ad)
\]
is a bijection, where the limit runs over the category of finite separable extensions of $K$ in $\Ksep$. 

Let $\kbar$ be the residue field of $\Ksep$, which is an algebraic closure of $k$. By the reduced fiber theorem and replacing $K$ with a sufficiently large finite separable extension, we may assume that $\mathring{A}$ is topologically of finite type over $\okey$ and $\mathring{A}\otimes_{\okey}\kbar$ is reduced. Then we have $(A\otimes_K L)^\circ=\mathring{A}\otimes_{\okey}\oel$ for any finite separable extension $L/K$. We may also assume that the inverse system $\{\pi_0(X^\ad_L)\}_{L/K}$ is constant. This implies that for any connected component $C$ of $X^\ad_K$, its inverse image $p_{L,K}^{-1}(C)$ by the natural projection $p_{L,K}:X^\ad_L\to X^\ad_K$ is a connected component of $X^\ad_L$. 

We claim the equality $(A\hat{\otimes}_K \bC)^\circ=\mathring{A}\hat{\otimes}_{\okey}\oc$. Indeed, let $\varpi$ be any non-zero element of the maximal ideal of $\okey$ and consider the exact sequence
\[
0 \to \mathring{A}\otimes_{\okey}\oc \overset{\times \varpi^l}{\to} \mathring{A}\otimes_{\okey}\oc \to \mathring{A}\otimes_{\okey}(\oc/\varpi^l\oc)\to 0
\]
for any positive integer $l$. Since the $\oc$-algebra $\mathring{A}\otimes_{\okey}\oc$ is $\varpi$-torsion free, the $\varpi$-adic topology on the ring $\mathring{A}\otimes_{\okey}\oc$ of the middle term of the sequence induces the $\varpi$-adic topology on the ring of its left term. Taking the $\varpi$-adic completion, we have an exact sequence
\[
0 \to \mathring{A}\hat{\otimes}_{\okey}\oc \overset{\times \varpi^l}{\to} \mathring{A}\hat{\otimes}_{\okey}\oc \to \mathring{A}\otimes_{\okey}(\oc/\varpi^l\oc)\to 0
\]
and thus the ring $\mathring{A}\hat{\otimes}_{\okey}\oc$ is torsion free. Moreover, since $\mathring{A}$ is topologically of finite type, we can choose an $\okey$-algebra surjection $\okey\langle X_1,\ldots,X_n\rangle \to \mathring{A}$. Via the natural surjection
\[
\okey\langle X_1,\ldots,X_n\rangle\otimes_{\okey}\oc \to \mathring{A}\otimes_{\okey}\oc,
\]
the $\varpi$-adic topology on the right-hand side coincides with the quotient topology of the $\varpi$-adic topology on the left-hand side. Thus we obtain a surjection
\[
\oc\langle X_1,\ldots,X_n\rangle\simeq \okey\langle X_1,\ldots,X_n\rangle\hat{\otimes}_{\okey}\oc \to \mathring{A}\hat{\otimes}_{\okey}\oc.
\]
By the above exact sequence, the special fiber of $\mathring{A}\hat{\otimes}_{\okey}\oc$ is isomorphic to the $\kbar$-algebra $\mathring{A}\otimes_{\okey}\kbar$, which is reduced by assumption. Then \cite[Proposition 1.1]{BLR4} implies the equality $(A\hat{\otimes}_K \bC)^\circ=\mathring{A}\hat{\otimes}_{\okey}\oc$ and the claim follows.

By this claim, we have
\[
X^\ad_\bC=X^\ad_L\times_{\Spa(L,\oel)}\Spa(\bC,\oc)
\]
for any finite separable extension $L/K$ and \cite[Lemma 3.9 (i)]{Hu_Fo} implies that the projection $p_{\bC,L}:X^\ad_\bC\to X^\ad_L$ is a surjection.

Note that for any affinoid ring $(R,S)$, there exists a natural homeomorphism
\[
\Spa(\hat{R},\hat{S})\to \Spa(R,S)
\]
preserving rational subsets, where $\hat{R}$ and $\hat{S}$ are the completions of $R$ and $S$, respectively (\cite[Proposition 3.9]{Hu_C}). Thus we have a homeomorphism 
\[
X_\bC^\ad\to \Spa(A\otimes_K \Ksep, \mathring{A}\otimes_{\okey}\cO_{\Ksep})
\]
preserving rational subsets, where the topology of the ring $A\otimes_K \Ksep$ is given by the $\varpi$-adic topology of the subring $\mathring{A}\otimes_{\okey}\cO_{\Ksep}$ for any non-zero element $\varpi$ in the maximal ideal of $\okey$. From this homeomorphism, we see that any rational subset of $X_\bC^\ad$ is the inverse image of a rational subset of $X_L^\ad$ for some finite separable extension $L$ of $K$.

Let $C$ be any connected component of $X^\ad_K$. To prove the lemma, it is enough to show that the inverse image $p^{-1}_{\bC,K}(C)$ is connected. Note that $C$ is a rational subset. Suppose that we have a decomposition $p^{-1}_{\bC,K}(C)=V_1\coprod V_2$ into the disjoint union of non-trivial open subsets. Since $p_{\bC,K}^{-1}(C)$ is also a rational subset, the open subsets $V_i$ are quasi-compact and thus are finite unions of rational subsets. This implies that the open subsets $V_i$ are the inverse images of some open subsets of $X^\ad_L$ for a sufficiently large finite separable extension $L$ of $K$. Since the projection $p_{\bC,L}$ is a surjection and $p^{-1}_{L,K}(C)$ is connected, the lemma follows.
\end{proof}



\section{Comparison of geometric connected components for affinoids of different characteristics}\label{comp}



\subsection{Lifts of truncated discrete valuation rings}\label{subseclift}

Let $A$ be a truncated discrete valuation ring of length $m$ (\cite[Subsection 1.1]{De}, \cite[Section 2]{HT}) with residue field $k$ of characteristic $p>0$. We fix a uniformizer $\bar{\pi}$ of $A$.

Let us consider a complete discrete valuation field $K$ and a surjective local homomorphism $\iota: \okey \to A$. We refer to such a pair $(K,\iota)$ as a lift of the truncated discrete valuation ring $A$. Note that a lift of $A$ always exists (\cite[Subsection 1.1]{De}). Let us fix a uniformizer $\pi$ of $K$ satisfying $\iota(\pi)=\bar{\pi}$. The map $\iota$ induces an isomorphism $\okey/(\pi^m)\simeq A$. We identify the residue field of $K$ with $k$ using this isomorphism. Let $e=e(K)$ be as in Section \ref{intro}. We also fix an algebraic closure $\Kbar$ of $K$ and extend the valuation $|\cdot|$ of $K$ to $\Kbar$. The residue field of $\Kbar$ is denoted by $\kbar$. Let $\bC$ be the completion of $\Kbar$ and $\mathfrak{m}_\bC$ be the maximal ideal of the valuation ring $\oc$. The field $\bC$ is a perfectoid field in the sense of \cite[Definition 3.1]{Sch}. We let $\fR$ denote its tilt. 

Suppose $pA=0$. Then we have $m\leq e$ (for the case of $\ch(K)=p$, this means that we take an arbitrarily large $e$ so that this inequality holds) and the field $\fR$ can be constructed using $m$ as follows. Let $\oR$ be the inverse limit ring
\[
\oR=\varprojlim_\Phi(\okbar/(\pi^m)\gets\okbar/(\pi^m)\gets\cdots),
\]
where $\Phi$ means that the transition maps are given by $x\mapsto x^p$. We have a natural multiplicative map
\[
\oR\to \oc,\ x\mapsto x^\sharp
\]
sending $x=(x_0,x_1,\ldots)\in \oR$ to the limit $x^\sharp=\lim_{l\to \infty} \hat{x}_l^{p^l}$ in the ring $\oc$, where $\hat{x}_l\in \oc$ is a lift of $x_l$. Note that the element $x^\sharp$ is independent of the choice of lifts $\hat{x}_l$ and the equality 
\[
x^\sharp\bmod \pi^m=\prjt_0(x)
\]
holds. The ring $\oR$ is a complete valuation ring of rank one and characteristic $p$ with algebraically closed fraction field $\fR$ whose valuation is defined by $|x|=|x^\sharp|$, and the map $(\cdot)^\sharp$ extends to a natural multiplicative map $(\cdot)^\sharp:\fR\to \bC$. If $K$ is of characteristic $p$, then the map $(\cdot)^\sharp$ gives an isomorphism of valuation fields $\fR\to \bC$. The maximal ideal of the valuation ring $\oR$ is denoted by $\mathfrak{m}_{\fR}$.

We fix a system $(\pi_l)_{l\geq 0}$ of $p$-power roots of $\pi$ in $\Kbar$ such that $\pi_0=\pi$ and $\pi_{l+1}^p=\pi_l$. The system defines an element $\upi=(\pi_0,\pi_1,\ldots)$ of the ring $\oR$ satisfying $\upi^\sharp=\pi$. 

Suppose also that $A$ is endowed with a $k$-algebra structure such that the diagram
\[
\xymatrix{
k \ar[r]\ar[rd]_{\id} & A\ar[d]\\
& k
}
\]
commutes, where the vertical arrow is the reduction map. Then the ring $A$ also lifts to a complete discrete valuation ring of equal characteristic $p$. Namely, the map $\boldsymbol{\iota}:k[[u]]\to A$ sending $u$ to $\bar{\pi}$ gives an isomorphism of $k$-algebras $k[[u]]/(u^m)\simeq A$. We put $F=k((u))$. Then the pair $(F, \boldsymbol{\iota})$ defines a lift of $A$. For any algebraic closure $\bar{F}$ of $F$, we extend the $u$-adic valuation $|\cdot|$ of $F$ to $\bar{F}$ naturally. We normalize it as $|u|=|\pi|$.



\subsection{Tubular neighborhoods of equations over $A$}\label{subsectub}

Let $A$ be a truncated discrete valuation ring of length $m$ and $(K,\iota)$ be a lift of $A$. Let $n$ be a positive integer and $\bar{f}=\{\bar{f}_1,\ldots,\bar{f}_r\}$ be a finite subset of the polynomial ring $A[X_1,\ldots,X_n]$. Let $f_i$ be a lift of $\bar{f}_i$ by the surjection
\[
\okey[X_1,\ldots,X_n]\to A[X_1,\ldots,X_n]
\]
induced by $\iota:\okey \to A$. For any $j=(j_1,\ldots j_r)\in (\bQ\cap (0,m])^r$, let us write $j_i=k_i/l_i$ with positive integers $k_i$ and $l_i$. Put
\[
\sB^j_K(\bar{f},n)=K\langle X_1,\ldots,X_n\rangle \langle \frac{f_1^{l_1}}{\pi^{k_1}},\ldots,\frac{f_r^{l_r}}{\pi^{k_r}} \rangle.
\]
This ring is a $K$-affinoid algebra independent of the choices of presentations $j_i=k_i/l_i$ and lifts $f_i$. Then we define the $j$-th tubular neighborhood $X^{j}_K(\bar{f},n)$ of $\bar{f}$ with respect to $n$ along the lift $(K,\iota)$ to be the following rational subdomain of the $n$-dimensional rigid analytic unit polydisc
\begin{align*}
X^{j}_K(\bar{f},n)&=\Spv(\sB^j_K(\bar{f},n))\\
&=\{x\in\Spv(K\la X_1,\ldots,X_n\ra)\mid |f_i(x)|\leq |\pi|^{j_i} \text{ for any }i\}.
\end{align*}

Suppose that $pA=0$ and $A$ is endowed with a $k$-algebra structure which gives a section of the reduction map. Let $(F,\boldsymbol{\iota})$ be the lift of $A$ as above. Then we can construct a similar tubular neighborhood of $\bar{f}$ on the side of $F$: Choose a lift $\mathbf{f}_i$ of $\bar{f}_i$ by the surjection
\[
k[[u]][X_1,\ldots,X_n] \to A[X_1,\ldots,X_n]
\]
induced by $\boldsymbol{\iota}: k[[u]]\to A$. We define the $j$-th tubular neighborhood $X^{j}_F(\bar{f},n)$ of $\bar{f}$ with respect to $n$ along the lift $(F,\boldsymbol{\iota})$ by
\begin{align*}
\sB^j_F(\bar{f},n)&=F\langle X_1,\ldots,X_n\rangle \langle \frac{\mathbf{f}_1^{l_1}}{u^{k_1}},\ldots,\frac{\mathbf{f}_r^{l_r}}{u^{k_r}} \rangle,\\
X^{j}_F(\bar{f},n)&=\Spv(\sB^j_F(\bar{f},n))\\
&=\{x\in\Spv(F\la X_1,\ldots,X_n\ra)\mid |\mathbf{f}_i(x)|\leq |u|^{j_i} \text{ for any }i\}.
\end{align*}
These are also independent of the choices of presentations $j_i=k_i/l_i$ and lifts $\mathbf{f}_i$. Note that the numbers of geometric connected components of these affinoid varieties are finite. 



\subsection{The case of perfect residue field}\label{subsecperf}

Now we assume $pA=0$ until the end of Section \ref{comp}. We also assume that the residue field $k$ of $A$ is perfect until the end of Subsection \ref{subseccomp}. Since $k$ is perfect, we have the unique inclusions $k\to A$ and $k\to \okey/(\pi^m)$ which are sections of the reduction maps by \cite[Chapitre II, \S 4, Proposition 8]{Se}, and the isomorphism $\okey/(\pi^m)\to A$ induced by $\iota$ is $k$-linear. Let $K^\mathrm{nr}$ be the maximal unramified extension of $K$ in $\Kbar$. Since the residue field of $K^\mathrm{nr}$ is $\kbar$ in this case, we also have the unique section $\kbar\to \cO_{K^\mathrm{nr}}/(\pi^m)$ of the reduction map. This gives an inclusion $[\cdot]: \kbar\to \okbar/(\pi^m)$ which is compatible with the map $k\to \okey/(\pi^m)$, and a natural inclusion
\[
\kbar \to \oR,\ x\mapsto ([x],[x^{1/p}],[x^{1/p^2}],\cdots).
\]
Then the map $(\cdot)^\sharp$ induces an isomorphism of $\kbar$-algebras
\[
\oR/(\upi^m) \simeq \okbar/(\pi^m).
\]

Consider the lift $(F,\boldsymbol{\iota})$ of $A$. The map $u\mapsto \upi$ and the natural inclusion $k\to \oR$ define an inclusion $F\to \fR$, by which we consider $F$ as a subfield of $\fR$. By our normalization, the valuation $|\cdot|$ of $F$ coincides with the restriction of the valuation $|\cdot|$ of $\fR$ to the subfield $F$. We have a commutative diagram of $k$-algebras
\begin{equation}\label{iotadiagram}
\begin{gathered}
\xymatrix{
\oef \ar[r]\ar[d]_{\boldsymbol{\iota}} & \oR\ar[dd]^{\prjt_0} \\
A\ar[d]_{\iota^{-1}} & \\
\okey/(\pi^m) \ar[r] & \okbar/(\pi^m),
}
\end{gathered}
\end{equation}
since the left vertical composite sends $u$ to $\pi$.

We choose an algebraic closure $\bar{F}$ of $F$ as the algebraic closure of $F$ in $\fR$, and let $F^\sep$ be the separable closure of $F$ in $\fR$. The subfield $F^\sep$ is dense in $\fR$ and the absolute Galois group $G_F=\Gal(F^\sep/F)$ acts naturally on $\fR$. Put $K_\infty=\Ksep\cap (\cup_n K(\pi_n))$ and $G_{K_\infty}=\Gal(\Ksep/K_\infty)$. By the classical theory of fields of norms of Fontaine-Wintenberger (see \cite{W}), the inclusion $F\to \fR$ gives an isomorphism of groups
\[
G_{K_\infty}\simeq G_F
\]
which is compatible with the map $(\cdot)^\sharp: \fR\to \bC$. 



\subsection{A comparison theorem}\label{subseccomp}

In this subsection, we prove the following main theorem of this section.

\begin{thm}\label{comparison}
There exists an isomorphism of finite $G_{K_\infty}$-sets 
\[
\rho_{\bar{f},n}^{K,F}: \pi_0^\geom(X^{j}_K(\bar{f},n))\to \pi_0^\geom(X^{j}_F(\bar{f},n))
\]
via the isomorphism $G_{K_\infty}\simeq G_F$.
\end{thm}

\begin{proof}
Put $X=(X_1,\ldots,X_n)$. Let us consider the rings
\[
\oc[X^{1/p^l}]=\oc[X_1^{1/p^l},\ldots, X_n^{1/p^l}],\ \oc[X^{1/p^\infty}]=\oc[X_1^{1/p^\infty},\ldots, X_n^{1/p^\infty}]
\]
for any non-negative integer $l$ and their $\pi$-adic completions 
\[
\oc\la X^{1/p^l}\ra=\oc[X^{1/p^l}]^\wedge,\ \oc\la X^{1/p^\infty}\ra=\oc[X^{1/p^\infty}]^\wedge.
\]
We also put
\[
\bC\la X^{1/p^l}\ra=\oc\la X^{1/p^l}\ra[1/\pi],\ \bC\la X^{1/p^\infty}\ra=\oc\la X^{1/p^\infty}\ra[1/\pi].
\]
On the side of $F$, we write as
\[
\oR[X^{1/p^l}]=\oR[X_1^{1/p^l},\ldots, X_n^{1/p^l}],\ \oR[X^{1/p^\infty}]=\oR[X_1^{1/p^\infty},\ldots, X_n^{1/p^\infty}]
\]
and their $\upi$-adic completions as
\[
\oR\la X^{1/p^l}\ra=\oR[X^{1/p^l}]^\wedge,\ \oR\la X^{1/p^\infty}\ra=\oR[X^{1/p^\infty}]^\wedge.
\] 
Similarly, we put
\[
\fR\la X^{1/p^l}\ra=\oR\la X^{1/p^l}\ra[1/\upi],\ \fR\la X^{1/p^\infty}\ra=\oR\la X^{1/p^\infty}\ra[1/\upi].
\]
By \cite[Proposition 5.20]{Sch}, the ring $\bC\la X^{1/p^\infty}\ra$ is a perfectoid $\bC$-algebra with ring of power-bounded elements
\[
\bC\la X^{1/p^\infty}\ra^\circ=\oc\la X^{1/p^\infty}\ra
\]
and its tilt given by
\[
\bC\la X^{1/p^\infty}\ra^\flat=\fR\la X^{1/p^\infty}\ra,\ \bC\la X^{1/p^\infty}\ra^{\flat\circ}=\oR\la X^{1/p^\infty}\ra.
\]
Moreover, we also have a continuous multiplicative map
\[
(\cdot)^\sharp:\fR\la X^{1/p^\infty}\ra \to \bC\la X^{1/p^\infty}\ra,\ g\mapsto g^\sharp
\]
which is compatible with $(\cdot)^\sharp:\fR\to \bC$ and
induces an isomorphism
\[
\oR\la X^{1/p^\infty}\ra/(\upi^m) \to \oc \la X^{1/p^\infty}\ra/(\pi^m)
\]
(\cite[Proposition 5.17 and Lemma 6.2]{Sch}). 

\begin{lem}\label{sharp}
The map $(\cdot)^\sharp$ induces the natural isomorphism
\[
\oR\la X^{1/p^\infty}\ra/(\upi^m) \to \oc \la X^{1/p^\infty}\ra/(\pi^m)
\]
defined by $X_i^{1/p^l}\mapsto X_i^{1/p^l}$ over the isomorphism $\prjt_0:\oR/(\upi^m)\to \oc/(\pi^m)$.
\end{lem}
\begin{proof}
We basically follow the notation of \cite[Proposition 5.17]{Sch}. Put $R=\bC\la X^{1/p^\infty}\ra$ and $R'=\fR\la X^{1/p^\infty}\ra$. Let $\sigma: R'^{\circ}\to R^\circ/(\pi^m)$ be the composite of the natural surjection and the isomorphism in the lemma. Consider the localization functor $M\mapsto M^a$ from the category of $\oc$-modules ({\it resp.} $\oR$-modules) to that of almost $\oc$-modules ({\it resp.} almost $\oR$-modules) and its right adjoint $N\mapsto N_*$. The inverse limit ring $\varprojlim_\Phi R^\circ/(\pi^m)$ along the Frobenius homomorphism is endowed with a natural $\oR$-algebra structure. Put $A=R^{\circ a}$ and $A^{\flat'}=(\varprojlim_\Phi R^\circ/(\pi^m))^a$. Then, from the proof of \cite[Proposition 5.17 and 5.20]{Sch}, there exists a unique isomorphism of almost $\oR$-algebras $\Psi: R'^{\circ a}\to A^{\flat'}$ which makes the following diagram commutative.
\[
\xymatrix{
R'^{\circ a} \ar[r]^{\Psi}\ar[rd]_{\sigma^a} & A^{\flat'}\ar[d]^{\prjt_0}\\
 & (R^\circ/(\pi^m))^a
}
\]
The map $g\mapsto g^\sharp\bmod \pi^m$ is the composite of $\Psi_*:R'^{\circ}=(R'^{\circ a})_*\to (A^{\flat'})_*$ and the natural map
\begin{align*}
(A^{\flat'})_*&=\Hom_{\mathcal{O}_{\bC^{\flat}}^a}(\mathcal{O}_{\bC^{\flat}}^a,(\varprojlim_\Phi R^\circ/(\pi^m))^a)\simeq \varprojlim_\Phi \Hom_{\mathcal{O}^a_{\bC}}(\mathcal{O}^a_{\bC}, (R^\circ/(\pi^m))^a)\\
&\overset{\prjt_0}{\to} \Hom_{\mathcal{O}^a_{\bC}}(\mathcal{O}^a_{\bC}, (R^\circ/(\pi^m))^a)=(A/(\pi^m))_*
\end{align*}
whose image is contained in the subring $A_*/(\pi^m)\simeq R^\circ/(\pi^m)$. The above map $(A^{\flat'})_*\to (A/(\pi^m))_*$ is equal to the map
\[
\Hom_{\mathcal{O}_{\bC^{\flat}}}(\mathfrak{m}_{\fR},\varprojlim_\Phi R^\circ/(\pi^m))\to \Hom_{\oc}(\mathfrak{m}_{\bC}, R^\circ/(\pi^m))
\]
defined by $h\mapsto (\varepsilon\mapsto \prjt_0(h(\varepsilon^\flat)))$, where $\varepsilon^\flat$ is any element of $\oR$ which is sent by the zeroth projection $\prjt_0:\oR\to \oc/(\pi^m)$ to the element $\varepsilon\bmod \pi^m$. By this description, we see that for any $g\in R'^{\circ}$, its image $g^\sharp\bmod \pi^m\in (A/(\pi^m))_*$ coincides with the map
\[
\delta\varepsilon \mapsto \prjt_0(\delta^\flat \Psi(\varepsilon^\flat g))=\delta\varepsilon\sigma(g),
\]
where $\delta,\varepsilon \in \mathfrak{m}_{\bC}$. Namely, it is the map $\varepsilon\mapsto \varepsilon\sigma(g)$.
Let $\widehat{\sigma(g)}$ be a lift of $\sigma(g)$ to the ring $R^\circ$.
Then the map $\mathfrak{m}_{\bC}\to R^\circ$ defined by $\varepsilon\mapsto \varepsilon\widehat{\sigma(g)}$ is an element of $A_*$ which is sent to $g^\sharp\bmod \pi^m\in (A/(\pi^m))_*$ by the natural map $A_*\to (A/(\pi^m))_*$. Since the isomorphism
\[
R^\circ \to A_*=\Hom_{\mathcal{O}^a_{\bC}}(\mathcal{O}^a_{\bC}, A)=\Hom_{\oc}(\mathfrak{m}_\bC, R^\circ)
\]
is given by $h \mapsto (\varepsilon \mapsto \varepsilon h)$, we conclude the equality $g^\sharp\bmod \pi^m=\sigma(g)$.
\end{proof}
Then the commutative diagram (\ref{iotadiagram}) and Lemma \ref{sharp} give the following corollary.

\begin{cor}\label{congruence}
The congruence
\[
\mathbf{f}_i^\sharp \equiv f_i\bmod \pi^m
\]
holds in the ring $\oc\la X^{1/p^\infty}\ra$.
\end{cor}

Consider the adic spaces
\begin{align*}
X_{\bC,l}^\ad&=\Spa(\bC\la X^{1/p^l}\ra, \oc\la X^{1/p^l}\ra),\ X_{\bC,\infty}^\ad=\Spa(\bC\la X^{1/p^\infty}\ra, \oc\la X^{1/p^\infty}\ra),\\
X_{\fR,l}^\ad&=\Spa(\fR\la X^{1/p^l}\ra, \oR\la X^{1/p^l}\ra),\ X_{\fR,\infty}^\ad=\Spa(\fR\la X^{1/p^\infty}\ra, \oR\la X^{1/p^\infty}\ra).
\end{align*}
By \cite[Theorem 6.3]{Sch}, there exists a homeomorphism $\tau: X_{\bC,\infty}^\ad \to X_{\fR,\infty}^\ad$ preserving rational subsets of both sides and satisfying $|g(\tau(x))|=|g^\sharp(x)|$ for any $x\in X_{\bC,\infty}^\ad$ and $g\in \fR\la X^{1/p^\infty}\ra$. Here $|\cdot(x)|$ denotes the continuous valuation associated to the point $x$.
We have a diagram
\begin{equation}\label{diagramX}
\begin{gathered}
\xymatrix{
X_{\bC,\infty}^\ad \ar[r]^-{\tau}_-{\sim}\ar[d]_{p_{\infty,l}} & X_{\fR,\infty}^\ad\ar[d]^{p_{\infty,l}^\flat}\\
X_{\bC,l}^\ad \ar[d]_{p_{l,0}} & X_{\fR,l}^\ad\ar[d]^{p_{l,0}^\flat}\\
X_{\bC,0}^\ad & X_{\fR,0}^\ad,
}
\end{gathered}
\end{equation}
where the vertical arrows are the natural projections. 

\begin{lem}\label{phomeo}
The projection $p_{l,l'}:X^{\ad}_{\bC,l}\to X^{\ad}_{\bC,l'}$ is a continuous open surjection for any $l,l'\in \bZ_{\geq 0}\cup \{\infty\}$ satisfying $l\geq l'$. Moreover, if $K$ is of characteristic $p$, then $p_{l,l'}$ is a homeomorphism.
\end{lem}
\begin{proof}
We may assume $l'=0$. Let $l$ be a non-negative integer. The maps
\[
\bC\la X\ra \to \bC\la X^{1/p^l}\ra,\ \oc\la X\ra \to \oc\la X^{1/p^l}\ra
\]
are flat and finitely presented, and also radicial if $\ch(K)=p$. We also see that the integral domain $\oc\la X^{1/p^l}\ra$ is integrally closed. By \cite[Lemma 1.7.9]{Hu_Et}, the continuous map $p_{l,0}$ is open. Furthermore, the map
\[
\Spec(\bC\la X^{1/p^l}\ra)\to \Spec(\bC\la X\ra)
\]
is a surjection, and also a homeomorphism if $\ch(K)=p$. Take $x\in X^\ad_{\bC,0}$. Let $\frp_x$ be the prime ideal of $\bC\la X\ra$ defined by
\[
\frp_x=\{f\in \bC\la X\ra\mid|f(x)|=0\}
\]
and $\kappa(x)$ be its residue field. Let $\frq$ be a prime ideal of $\bC\la X^{1/p^l}\ra$ above the prime ideal $\frp_x$, which is unique if $\ch(K)=p$. Then there exists a valuation on the residue field $\kappa(\frq)$ of $\frq$ whose restriction to $\kappa(x)$ is equivalent to the valuation $|\cdot(x)|$, and it is unique up to equivalence if $\ch(K)=p$, since in the latter case the residue field $\kappa(\frq)$ is a purely inseparable extension of $\kappa(x)$. We can show that this valuation defines a point of $X_{\bC,l}^\ad$ above $x$ and that such a point is unique if $\ch(K)=p$.
Hence the map $p_{l,0}$ is a continuous open surjection, and also a homeomorphism if $\ch(K)=p$.

Next we treat the case of $p_{\infty,0}$. By the equality
\[
\bC[X^{1/p^\infty}]=\varinjlim_l \bC[X^{1/p^l}]
\]
and \cite[Proposition 3.9]{Hu_C}, we have the commutative diagram
\[
\xymatrix{
X^{\ad}_{\bC,\infty}\ar[r]\ar[d] & \Spa(\bC[X^{1/p^\infty}],\oc[X^{1/p^\infty}])\ar[d]\\
X^{\ad}_{\bC,l}\ar[r]\ar[d] & \Spa(\bC[X^{1/p^l}],\oc[X^{1/p^l}])\ar[d]\\
X^{\ad}_{\bC,0}\ar[r] & \Spa(\bC[X],\oc[X])
}
\]
whose horizontal arrows are homeomorphisms preserving rational subsets. Hence, by extending valuations as above, we see that the continuous map $p_{\infty, l}$ is a surjection for any $l$ and any rational subset of $X^\ad_{\bC,\infty}$ is the inverse image of a rational subset of $X^\ad_{\bC,l}$ for some non-negative integer $l$. This proves the first assertion. If $\ch(K)=p$, we also see that the map $p_{\infty, l}$ is a bijection and the second assertion follows.
\end{proof}

Now we put
\begin{align*}
\sB^j_{\bC}(\bar{f},n)&=\sB^j_K(\bar{f},n)\hat{\otimes}_K \bC,\ X^{j,\ad}_{\bC}(\bar{f},n)=\Spa(\sB^j_{\bC}(\bar{f},n), \sB^j_{\bC}(\bar{f},n)^\circ),\\
\sB^j_{\fR}(\bar{f},n)&=\sB^j_F(\bar{f},n)\hat{\otimes}_F \fR,\ X^{j,\ad}_{\fR}(\bar{f},n)=\Spa(\sB^j_{\fR}(\bar{f},n), \sB^j_{\fR}(\bar{f},n)^\circ).
\end{align*}
Then we have the equalities
\begin{align*}
X^{j,\ad}_{\bC}(\bar{f},n)&=\{x \in X^\ad_{\bC,0}\mid |f_i(x)|\leq |\pi(x)|^{j_i}\text{ for any }i\},\\
X^{j,\ad}_{\fR}(\bar{f},n)&=\{x \in X^\ad_{\fR,0}\mid |\mathbf{f}_i(x)|\leq |u(x)|^{j_i}\text{ for any }i\},
\end{align*}
where $f_i$ and $\mathbf{f}_i$ are the lifts of $\bar{f}_i$ as before. By Lemma \ref{connCp}, we have natural bijections
\begin{align*}
\pi_0(X^{j,\ad}_{\bC}(\bar{f},n))&\overset{\sim}{\to}  \pi_0^\geom(X^{j}_K(\bar{f},n)),\\
\pi_0(X^{j,\ad}_{\fR}(\bar{f},n))&\overset{\sim}{\to}  \pi_0^\geom(X^{j}_F(\bar{f},n))
\end{align*}
which are compatible with the natural Galois action. Hence we are reduced to constructing a natural isomorphism of $G_{K_\infty}$-sets
\[
\pi_0(X^{j,\ad}_{\bC}(\bar{f},n))\to \pi_0(X^{j,\ad}_{\fR}(\bar{f},n)).
\]

For any $l\in \bZ_{\geq 0}\cup\{\infty\}$, set $X^{j,\ad}_{\bC,l}(\bar{f},n)$ to be the inverse image of the rational subset $X^{j,\ad}_{\bC}(\bar{f},n)\subseteq X^\ad_{\bC,0}$ by the natural projection $p_{l,0}: X^\ad_{\bC,l}\to X^\ad_{\bC,0}$. This is the rational subset of $X^\ad_{\bC,l}$ defined by
\[
\{x\in X^\ad_{\bC,l}\mid |f_i(x)|\leq |\pi(x)|^{j_i}\text{ for any }i\}.
\]

\begin{lem}\label{KEY}
The rational subset $X_{\bC,\infty}^{j,\ad}(\bar{f},n)$ of $X^\ad_{\bC,\infty}$ is the inverse image of the rational subset $X_{\fR,\infty}^{j,\ad}(\bar{f},n)$ of $X^\ad_{\fR,\infty}$ by the homeomorphism $\tau$.
\end{lem}
\begin{proof}
By the relation $|g(\tau(x))|=|g^\sharp(x)|$, the inverse image in the lemma is the rational subset
\[
\{ x\in X^\ad_{\bC,\infty}\mid |\mathbf{f}_i^\sharp(x)|\leq |u^\sharp(x)|^{j_i}\text{ for any }i\}
\]
of the adic space $X^\ad_{\bC,\infty}$. Note the equality $u^\sharp=\pi$. By Corollary \ref{congruence} and the assumption $j_i\leq m$, we obtain the equivalence
\[
|\mathbf{f}_i^\sharp(x)|\leq |\pi(x)|^{j_i} \Leftrightarrow |f_i(x)|\leq |\pi(x)|^{j_i}
\]
and the lemma follows.
\end{proof}

Therefore, the diagram (\ref{diagramX}) induces a diagram
\begin{equation}\label{diagramj}
\begin{gathered}
\xymatrix{
X_{\bC,\infty}^{j,\ad}(\bar{f},n) \ar[r]^-{\tau}_-{\sim}\ar[d]_{p_{\infty,l}} & X_{\fR,\infty}^{j,\ad}(\bar{f},n)\ar[d]^{p_{\infty,l}^\flat}\\
X_{\bC,l}^{j,\ad}(\bar{f},n) \ar[d]_{p_{l,0}} & X_{\fR,l}^{j,\ad}(\bar{f},n)\ar[d]^{p_{l,0}^\flat}\\
X_{\bC}^{j,\ad}(\bar{f},n) & X_{\fR}^{j,\ad}(\bar{f},n),
}
\end{gathered}
\end{equation}
where $\tau$ is a homeomorphism.

\begin{lem}\label{pi0toinf}
The natural projections induce isomorphisms
\begin{align*}
\pi_0(X^{j,\ad}_{\bC,\infty}(\bar{f},n))&\to \pi_0(X^{j,\ad}_{\bC,l}(\bar{f},n))\to \pi_0(X^{j,\ad}_{\bC}(\bar{f},n)),\\
\pi_0(X^{j,\ad}_{\fR,\infty}(\bar{f},n))&\to \pi_0(X^{j,\ad}_{\fR,l}(\bar{f},n)) \to \pi_0(X^{j,\ad}_{\fR}(\bar{f},n))
\end{align*}
of $G_K$-sets ({\it resp.} $G_F$-sets) for any $l$. 
\end{lem}
\begin{proof}
We may only consider the case over $\bC$. Since the projections are continuous and compatible with the natural $G_K$-action, the maps are well-defined. It is enough to show the bijectivity. If $K$ is of characteristic $p$, then this follows from Lemma \ref{phomeo}. 

Suppose that $K$ is of mixed characteristic. Note that the ring $\sB^j_{\bC}(\bar{f},n)$ is Noetherian and \cite[Theorem 2.2]{Hu_Fo} (or Lemma \ref{adrigconn}) implies that the number of connected components of $X^{j,\ad}_{\bC}(\bar{f},n)$ is finite. Moreover, each of its connected components is a rational subset. Since $p_{l,0}$ is a surjection for any $l\in \bZ_{\geq 0}\cup\{\infty\}$, it suffices to show that, for any connected component $C$ of $X^{j,\ad}_{\bC}(\bar{f},n)$, the inverse image $p_{l,0}^{-1}(C)$ is connected.

Suppose that we have a decomposition $p_{l,0}^{-1}(C)=V_1\coprod V_2$ into the disjoint union of non-trivial open subsets. Since $p_{l,0}^{-1}(C)$ is also a rational subset, the open subsets $V_i$ are quasi-compact and thus are finite unions of rational subsets. For the case of $l=\infty$, this implies that the open subsets $V_i$ are the inverse images of some open subsets of $X^{j,\ad}_{\bC,l'}(\bar{f},n)$ for a sufficiently large non-negative integer $l'$. Since the projection $p_{\infty,l'}$ is a surjection, this shows that we may assume $l\in \bZ_{\geq 0}$.

Let $l$ be a non-negative integer. Since the map $p_{l,0}$ is a continuous open surjection, the images $p_{l,0}(V_i)$ are non-trivial quasi-compact open subsets covering the connected component $C$ and thus they would meet each other. By \cite[Corollary 4.2]{Hu_C}, the intersection of these images has a point defined by the map $X_i\mapsto x_i$ with some $x_i\in \oc$. Thus we reduce ourselves to showing that, for any such classical point $x=(x_1,\ldots,x_n)\in X^{j,\ad}_{\bC}(\bar{f},n)$, any two points $y,y'\in p^{-1}_{l,0}(x)$ are contained in the same connected component of $X_{\bC,l}^{j,\ad}(\bar{f},n)$.

Consider the rational subset
\begin{align*}
U&=X^{\ad}_{\bC,0}(\frac{X_1-x_1,\ldots, X_n-x_n}{\pi^m})\\
&=\{z\in X^{\ad}_{\bC,0}\mid |(X_i-x_i)(z)|\leq |\pi(z)|^m\text{ for any }i\}
\end{align*}
of $X^{\ad}_{\bC,0}$ containing $x$. Since $x$ satisfies the inequality
\[
| f_i(x)|\leq |\pi|^{j_i}
\]
for any $i$, our assumption $j_i\leq m$ implies that any point $z\in U$ also satisfies the inequality and $U$ is contained in $X_{\bC}^{j,\ad}(\bar{f},n)$. Then the inverse image $p_{l,0}^{-1}(U)$ is the rational subset 
\[
\{z\in X^{\ad}_{\bC,l}\mid |(X_i-x_i)(z)|\leq |\pi(z)|^m\text{ for any }i\}
\]
of $X^{\ad}_{\bC,l}$ containing $y$ and $y'$ which is contained in $X_{\bC,l}^{j,\ad}(\bar{f},n)$.

\begin{lem}\label{Uconn}
For any $z\in X^{\ad}_{\bC,l}$, any $a\in \oc$ and any positive rational number $j'$ satisfying 
\[
|\pi|^{p^l e/(p^{l-1}(p-1))}|a|^{p^l} \leq |\pi|^{j'}, 
\]
we have the equivalence
\[
|(X_i-a^{p^l})(z)|\leq |\pi(z)|^{j'} \Leftrightarrow |(X_i^{1/p^l}-a)(z)|\leq |\pi(z)|^{j'/p^l}.
\]
\end{lem}
\begin{proof}
We let $\zeta_{p^l}$ denote a primitive $p^l$-th root of unity in $\bC$. Then we have 
\[
|(X_i-a^{p^l})(z)|=\prod_{s=0}^{p^l-1} |(X_i^{1/p^l}-a\zeta_{p^l}^s)(z)|.
\]
Suppose that the inequality 
\[
|(X_i^{1/p^l}-a\zeta_{p^l}^s)(z)|\leq |\pi(z)|^{j'/p^l}
\]
holds for some $s$. Then the assumption on $j'$ implies
\[
|(X_i^{1/p^l}-a\zeta_{p^l}^{s'})(z)|=|(X_i^{1/p^l}-a\zeta_{p^l}^{s}+a(\zeta_{p^l}^{s}-\zeta_{p^l}^{s'}))(z)|\leq |\pi(z)|^{j'/p^l}
\]
for any other $s'$. This shows the implication of one direction. Conversely, if $|(X_i-a^{p^l})(z)|\leq |\pi(z)|^{j'}$, then
\[
\min_s |(X_i^{1/p^l}-a\zeta_{p^l}^s)(z)|\leq |\pi(z)|^{j'/p^l}
\]
and the other direction also follows from the above claim.
\end{proof}

Since $m\leq e$, Lemma \ref{Uconn} shows that $p_{l,0}^{-1}(U)$ is equal to the rational subset
\[
\{z\in X^{\ad}_{\bC,l}\mid |(X_i^{1/p^l}-x_i^{1/p^l})(z)|\leq |\pi(z)|^{m/p^l}\text{ for any }i\},
\]
which is a polydisc and thus it is connected by \cite[Theorem 2.2]{Hu_Fo} (or Lemma \ref{adrigconn}). Hence the two points $y,y'$ are contained in the same connected component of $X_{\bC,l}^{j,\ad}(\bar{f},n)$.
\end{proof}

\begin{lem}\label{tauconn}
The homeomorphism $\tau$ induces an isomorphism of finite $G_{K_\infty}$-sets
\[
\pi_0(X^{j,\ad}_{\bC,\infty}(\bar{f},n)) \to \pi_0(X^{j,\ad}_{\fR,\infty}(\bar{f},n))
\]
via the isomorphism $G_{K_\infty}\simeq G_F$.
\end{lem}
\begin{proof}
By Lemma \ref{KEY}, the homeomorphism $\tau$ induces a bijection of the sets in the lemma. It is enough to show that this is compatible with the $G_{K_\infty}$-action. Note that the action $\sigma^*$ of any element $\sigma\in G_{K_\infty}$ on the adic space $X^{j,\ad}_{\bC,\infty}(\bar{f},n)$ is defined by the action of $\sigma^{-1}$ on the coefficients of the ring $\bC\la X^{1/p^\infty} \ra$ and similarly for $X^{j,\ad}_{\fR,\infty}(\bar{f},n)$. Every connected component $C$ of the adic space $X^{j,\ad}_{\bC,\infty}(\bar{f},n)$ is a rational subset which is the inverse image of a rational subset of $X^{\ad}_{\bC,0}$, as is shown in the proof of Lemma \ref{pi0toinf}. Thus $C$ contains a point $x$ defined by the map
\[
\bC\la X^{1/p^\infty} \ra \to \bC,\ X_i^{1/p^l}\mapsto x_{i,l}
\]
with some $x_{i,l}\in \oc$ satisfying $x_{i,l+1}^p=x_{i,l}$ for any $l$. It suffices to show for this $x$ that for any $\sigma\in G_{K_\infty}$, the points $\sigma^*(\tau(x))$ and $\tau(\sigma^*(x))$ are contained in the same connected component of $X^{j,\ad}_{\fR,\infty}(\bar{f},n)$. Note that we have
\begin{align*}
|g(\sigma^*(\tau(x)))|&=|\sigma^{-1}(g)(\tau(x))|=|(\sigma^{-1}(g))^\sharp(x)|,\\
|g(\tau(\sigma^*(x))|&=|g^\sharp(\sigma^*(x))|=|\sigma^{-1}(g^\sharp)(x)|
\end{align*}
for any $g\in \fR\la X^{1/p^\infty} \ra$.

The system $(x_{i,l})_{l\in \bZ_{\geq 0}}$ defines an element $\ux_i\in \oR$ for any $i$. Put $\ux=(\ux_1,\ldots, \ux_n)$ and $\ux^\sharp=(\ux_1^\sharp,\ldots, \ux_n^\sharp)$. By the definition of the map $(\cdot)^\sharp$, we have $\ux^\sharp=(x_{1,0},\ldots,x_{n,0})$. Consider the rational subset
\[
U=\{ z\in X^{\ad}_{\fR,\infty}\mid |(X_i-\sigma(\ux_i))(z)|\leq |u(z)|^{m}\text{ for any }i\}
\]
of the adic space $X^{\ad}_{\fR,\infty}$. This is the inverse image of a polydisc in $X^{\ad}_{\fR,0}$ by the projection $p_{\infty,0}^\flat$ and thus connected by Lemma \ref{phomeo}. By Lemma \ref{sharp}, we have congruences
\[
(\sigma^{-1}(X_i-\sigma(\ux_i)))^\sharp\equiv\sigma^{-1}(X_i-\sigma(x_{i,0}))\equiv \sigma^{-1}((X_i-\sigma(\ux_i))^\sharp)\bmod \pi^m
\]
in the ring $\oc\la X^{1/p^\infty}\ra$. Hence we have equivalences
\begin{align*}
&|(\sigma^{-1}(X_i-\sigma(\ux_i)))^\sharp(x)|\leq |(\sigma^{-1}(u))^\sharp(x)|^m \\
&\Leftrightarrow |(X_i-x_{i,0})(x)|\leq |\pi(x)|^m\\
&\Leftrightarrow |\sigma^{-1}((X_i-\sigma(\ux_i))^\sharp)(x)|\leq |\sigma^{-1}(u^\sharp)(x)|^m,
\end{align*}
which implies that the points $\sigma^*(\tau(x))$ and $\tau(\sigma^*(x))$ lie in $U$. 

On the other hand, let $\ux_i^{1/p^l}$ be the unique $p^l$-th root of $\ux_i$ in the perfect integral domain $\oR$. Then the map $X_i^{1/p^l}\mapsto \ux_i^{1/p^l}$ defines a point of the adic space $X^{\ad}_{\fR,\infty}$, which is also denoted by $\ux$. The commutative diagram (\ref{iotadiagram}) yields the congruence
\[
(\mathbf{f}_i(\ux))^\sharp \equiv f_i(x_{1,0},\ldots, x_{n,0})\bmod \pi^m
\]
in the ring $\oc$. Since $x\in X^{j,\ad}_{\bC,\infty}(\bar{f},n)$, we have the inequality
\[
|f_i(x_{1,0},\ldots,x_{n,0})|\leq |\pi|^{j_i}\text{ for any }i
\]
and the above congruence implies that the points $\ux$ and $\sigma^*(\ux)$ are contained in the rational subset $X^{j,\ad}_{\fR,\infty}(\bar{f},n)$. Note that the latter point is defined by the map $X_i^{1/p^l}\mapsto \sigma(\ux_i^{1/p^l})$. Thus we see that $U$ is contained in $X^{j,\ad}_{\fR,\infty}(\bar{f},n)$. This shows that the points $\sigma^*(\tau(x))$ and $\tau(\sigma^*(x))$ are contained in the same connected component of $X^{j,\ad}_{\fR,\infty}(\bar{f},n)$ and the lemma follows.
\end{proof}

By Lemma \ref{pi0toinf} and Lemma \ref{tauconn}, we have a diagram of bijections
\begin{equation}\label{diagramjpi0}
\begin{gathered}
\xymatrix{
\pi_0(X_{\bC,\infty}^{j,\ad}(\bar{f},n)) \ar[r]^-{\tau}_-{\sim}\ar[d]_{p_{\infty,0}}^{\wr} & \pi_0(X_{\fR,\infty}^{j,\ad}(\bar{f},n))\ar[d]^{p_{\infty,0}^\flat}_{\wr}\\
\pi_0(X_{\bC}^{j,\ad}(\bar{f},n)) & \pi_0(X_{\fR}^{j,\ad}(\bar{f},n)),
}
\end{gathered}
\end{equation}
where all arrows are compatible with the natural Galois action. This concludes the proof of Theorem \ref{comparison}.
\end{proof}

\begin{rmk}
The isomorphism $\rho_{\bar{f},n}^{K,F}$ of Theorem \ref{comparison} depends on the choices of a uniformizer of $A$, an algebraic closure $\Kbar$, a uniformizer $\pi$ and a system of its $p$-power roots $(\pi_l)_{l\in\bZ_{\geq 0}}$.
\end{rmk}



\subsection{The case of imperfect residue field}\label{subsecimperf}

Now we return to the situation of Subsection \ref{subsectub}. Namely, we consider a truncated discrete valuation ring $A$ of length $m$ with uniformizer $\bar{\pi}$ and residue field $k$, which may be imperfect. We also assume $pA=0$. We fix a $k$-algebra structure $k\to A$ which gives a section of the reduction map $A\to k$. Note that we can always find such a map by \cite[Th\'{e}or\`{e}me (19.6.1)]{EGA4-1}, since the extension $k/\bF_p$ is separable. Let $(K,\iota)$, $(F,\boldsymbol{\iota})$, $\bar{f}$, $f$ and $\mathbf{f}$ be as before.

We fix a Cohen ring $C(k)$ of $k$. Using \cite[Th\'{e}or\`{e}me (19.8.6) (i)]{EGA4-1}, we also fix a local homomorphism $C(k)\to \okey$ which makes the following diagram commutative.
\[
\xymatrix{
C(k)\ar[r]\ar[d] & \okey\ar[d]\\
k\ar[r] & A
}
\]

Suppose that $K$ is of characteristic zero. Then this local homomorphism is an injection and the $C(k)$-algebra $\okey$ is finite. By this fixed map, we consider $K_0=\Frac(C(k))$ as a subfield of $K$. The extension $K/K_0$ is a finite totally ramified extension.  

We fix a $p$-basis $\{\bar{b}_\lambda\}_{\lambda\in\Lambda}$ of $k$ and its lift $\{b_\lambda\}_{\lambda\in\Lambda}$ in $C(k)$. We also fix a system of $p$-power roots $(b_{\lambda,l})_{l\geq 0}$ of $b_{\lambda}$ in $\okbar$ satisfying $b_{\lambda,0}=b_\lambda$ and $b_{\lambda,l+1}^p=b_{\lambda,l}$. Let $K'_0$ be the completion of the discrete valuation field $\cup_{\lambda,l}K_0(b_{\lambda,l})$, which is naturally considered as a subfield of $\bC$. Then the extension $K_0'/K_0$ is of relative ramification index one and the residue field $k'$ of $K'_0$ is the perfect closure of $k$ in $\kbar$. Put $K'=K'_0K$, the composite field in $\bC$. This is a finite extension of $K'_0$, and $K'/K$ is an extension of complete discrete valuation fields of relative ramification index one satisfying $\cO_{K'}=\okey\otimes_{C(k)}\cO_{K'_0}$.

Next suppose that $K$ is of characteristic $p$. Then the map $C(k)\to \okey$ factors through $k$ and gives a $k$-algebra structure of $\okey$. We have an isomorphism of $k$-algebras $k[[u]]\to \okey$ sending $u$ to $\pi$. Let $k'$ be the perfect closure of $k$ in $\Kbar$ and $K'$ be the completion of the composite field $k'K$ in $\Kbar$. Then the field $K'$ is naturally isomorphic to $k'((u))$, and it is naturally considered as a subfield of $\bC$. Moreover, $K'/K$ is an extension of complete discrete valuation fields of relative ramification index one.

In both cases, let $K'^\sep$ be the separable closure of $K'$ in $\bC$ and put $K'_\infty=K'^\sep\cap (\cup_n K'(\pi_n))$ as before.

\begin{lem}\label{primaryext}
\begin{enumerate}
\item\label{primaryext-dense} The subfield $K'^\sep$ is dense in $\bC$.
\item\label{primaryext-Gal}
The natural map
\[
\Gal(K'^\sep/K')\to \Gal(\Ksep/\Ksep\cap K')
\]
is an isomorphism.
\item\label{primaryext-prim} If $\ch(K)=p$, then the extension $K'/K$ is primary. In particular, the map in (\ref{primaryext-Gal}) induces an isomorphism
\[
\Gal(K'^\sep/K')\simeq \Gal(\Ksep/K).
\]
\end{enumerate}
\end{lem}
\begin{proof}
Note that $K^\sep$ is a dense subfield of $\bC$. First suppose $\ch(K)=0$. Then Krasner's lemma implies $K'^\sep=K^\sep K'$ and the assertion (\ref{primaryext-dense}) follows. This equality also shows the assertion (\ref{primaryext-Gal}).

Next suppose $\ch(K)=p$. Let $(k'K)^\sep$ be the separable closure of $k'K$ in $\bC$. Krasner's lemma shows $K'^\sep=(k'K)^\sep K'\supseteq K^\sep$ and the assertion (\ref{primaryext-dense}) follows. Let $\sigma$ be an element of $\Gal(K'^\sep/K')$ satisfying $\sigma|_{K^\sep}=\id$. Take $x\in (k'K)^\sep$. Let $g(X)$ be its minimal polynomial over $k'K$ and write it as
\[
g(X)=X^N+a_1 X^{N-1}+\cdots+a_{N-1} X+a_N
\]
with some $a_i\in k'K$. Then there exists a non-negative integer $l$ satisfying $a_i^{p^l}\in K$ for any $i$. Thus $g(X)^{p^l}\in K[X]$ and $x^{p^l}\in K^\sep$. Hence we obtain $\sigma(x)=x$ and the map in the assertion (\ref{primaryext-Gal}) is an injection. Let $L$ be a finite Galois extension of $K^\sep\cap K'$ in $\Ksep$. Then we have $L\cap K'=K^\sep \cap K'$. This implies the isomorphism
\[
\Gal(L K'/K')\simeq \Gal(L/K^\sep \cap K')
\]
and the map in the assertion (\ref{primaryext-Gal}) is also a surjection.

Finally, we show that the extension $K'/K$ is primary. Since the algebraic extension $k'K/K$ is purely inseparable, it is enough to show that any finite separable extension $L/k'K$ in $K'$ coincides with $k'K$. Since the discrete valuation field $k'K$ is the union of finite extensions of $K$, it is Henselian. This implies that the valuation of $k'K$ uniquely extends to $L$, and thus the extended valuation is equal to the restriction of the valuation of $K'$. Since the relative ramification index and the residue degree of $L/k'K$ are both equal to one, we obtain $L=k'K$.
\end{proof}

Using the isomorphism of Lemma \ref{primaryext} (\ref{primaryext-Gal}), we consider the absolute Galois group $G_{K'}=\Gal(K'^\sep/K')$ as a subgroup of $G_K$.

We consider the $k$-algebra $A$ as a $C(k)$-algebra by the composite $C(k)\to k\to A$. Put $A'=A\otimes_k k'$. This ring can be also written as $A'=A\otimes_{\okey}\cO_{K'}$.
Indeed, this follows from the equality $A'=A\otimes_{C(k)}\cO_{K'_0}$ for $\ch(K)=0$ and 
\[
\cO_{k'K}/(\pi^m)=\varinjlim_{l/k}(\cO_{l K}/(\pi^m))=\varinjlim_{l/k}(\cO_{K}/(\pi^m)\otimes_k l)=\okey/(\pi^m)\otimes_k k'
\]
for $\ch(K)=p$, where the limit runs over the category of finite extensions inside $k'/k$. The map $\iota$ induces an isomorphism $\cO_{K'}/(\pi^m)\to A'$. Thus $A'$ is a truncated discrete valuation ring of length $m$ with perfect residue field $k'$ endowed with the induced map $k'\to A'$ giving a section of the reduction map, and also with the induced lift $\iota': \cO_{K'}\to A'$. Put $F'=k'((u))$. The field $F$ is considered as a subfield of $F'$ by the map $u\mapsto u$ and the natural inclusion $k\to k'$. Then the lift $\boldsymbol{\iota}$ also induces a lift $\boldsymbol{\iota}':\cO_{F'}=k'[[u]]\to A'$.
Hence we obtain the cartesian diagram
\[
\xymatrix{
k'[[u]] \ar[r]^{\boldsymbol{\iota}'} & A' & \cO_{K'} \ar[l]_{\iota'} \\
k[[u]]\ar[u] \ar[r]_{\boldsymbol{\iota}} & A\ar[u] & \cO_K\ar[l]^{\iota}. \ar[u]
}
\]
Let $\bar{f}'$ be the image of $\bar{f}$ by the map $A\to A'$. Similarly, let $f'$ and $\mathbf{f}'$ be the images of $f$ and $\mathbf{f}$ by the maps $\okey\to \cO_{K'}$ and $k[[u]]\to k'[[u]]$, respectively. We have the equality
\[
\iota'(f')=\bar{f}'=\boldsymbol{\iota}'(\mathbf{f}').
\]
Thus the sets of polynomials $\bar{f}'$, $f'$ and $\mathbf{f}'$ are also in the situation of Subsection \ref{subsectub}, for the truncated discrete valuation ring $A'$ with perfect residue field $k'$. Note that the extensions $K'/K$ and $F'/F$ are of relative ramification index one, and that formation of $j$-th tubular neighborhoods is compatible with the base change by any extension of relative ramification index one. Applying Lemma \ref{connCp}, we obtain natural bijections
\begin{align*}
\pi_0^\geom(X_K^j(\bar{f},n)) &\to \pi_0^\geom(X_{K'}^j(\bar{f}',n)),\\
\pi_0^\geom(X_F^j(\bar{f},n)) &\to \pi_0^\geom(X_{F'}^j(\bar{f}',n))
\end{align*}
which are compatible with the $G_{K'}$-action and the $G_{F'}$-action, respectively. Hence Theorem \ref{comparison} implies the following theorem.

\begin{thm}\label{comparisonimperf}
There exists an isomorphism of finite $G_{K'_\infty}$-sets
\[
\rho_{\bar{f},n}^{K,F}: \pi_0^\geom(X_K^j(\bar{f},n))\to \pi_0^\geom(X_F^j(\bar{f},n))
\]
via the isomorphism $G_{K'_\infty}\simeq G_{F'}$ of the classical theory of fields of norms which fits into the commutative diagram
\[
\xymatrix{
\pi_0^\geom(X_K^j(\bar{f},n))\ar[r]^{\rho^{K,F}_{\bar{f},n}}\ar[d]_{\wr} &\pi_0^\geom(X_F^j(\bar{f},n))\ar[d]^{\wr} \\
\pi_0^\geom(X_{K'}^j(\bar{f}',n))\ar[r]_{\rho^{K',F'}_{\bar{f}',n}} &\pi_0^\geom(X_{F'}^j(\bar{f}',n)),
}
\]
where $\rho_{\bar{f}',n}^{K',F'}$ is the isomorphism of Theorem \ref{comparison}.
\end{thm}
\begin{flushright}
\qed
\end{flushright}

\begin{rmk}
When $k$ is imperfect, the isomorphism $\rho_{\bar{f},n}^{K,F}$ depends on the choices of a $k$-algebra structure and a uniformizer of $A$, a $p$-basis and a Cohen ring $C(k)$ of $k$, an algebraic closure, a local map $C(k)\to \okey$, a uniformizer $\pi$, a lift $\{b_\lambda\}_{\lambda\in\Lambda}$ of the fixed $p$-basis of $k$ and compatible $p$-power roots of $\pi$ and $b_\lambda$ for $K$.
\end{rmk}





\section{Non-log ramification}\label{nonlog}

In the rest of the paper, we give applications of Theorem \ref{comparisonimperf} to ramification theory. Let $A$ be a truncated discrete valuation ring of length $m$, with residue field $k$ of characteristic $p>0$. We allow the case where $k$ is imperfect. We fix a uniformizer $\bar{\pi}$ of $A$. Let $B$ be a finite flat $A$-algebra. The aim of this section is to study ramification of the extension $B/A$, as in \cite{Ha} and \cite{HT}.

\subsection{Ramification theory over truncated discrete valuation rings}\label{Hatt}

First we briefly recall the construction of a ramification theory of $B/A$ for a fixed lift $(K,\iota)$ of $A$ given in \cite{Ha}. Let $(K,\iota)$ be a lift of $A$. Fix an algebraic closure $\Kbar$ and a uniformizer $\pi$ of $K$. We let $\kbar$ denote the residue field of $\Kbar$, as in Subsection \ref{subseclift}. Let $j\in\bQ\cap (0,m]$. Fix a system of finite generators $Z=(z_1,\ldots,z_n)$ of the $A$-algebra $B$. This defines a surjection of $A$-algebras
\[
A[X_1,\ldots,X_n]\to B,\ X_i\mapsto z_i.
\]
We let $\bar{I}_Z$ denote its kernel. Fix a system of finite generators $\bar{f}=\{\bar{f}_1,\ldots,\bar{f}_r\}$ of the ideal $\bar{I}_Z$. We put
\[
X^j_{K}(B,Z)=X^j_{K}(\bar{f},\sharp Z)
\]
with the notation of Section \ref{comp}, where we identify $j$ with the $r$-tuple $(j,\ldots,j)$. The $K$-affinoid variety $X^j_{K}(B,Z)$ is independent of the choice of a system of finite generators $\bar{f}$. This is referred to as the $j$-th tubular neighborhood of $B$ with respect to a system of finite generators $Z$ along the lift $(K,\iota)$.

We consider a question of functoriality of the finite $G_K$-set $\pi_0^\geom(X^j_{K}(B,Z))$.
Though this is done in \cite[Section 2]{Ha}, we present here detailed proofs of results stated in \cite{Ha} whose proofs are omitted, since we will use some of the omitted arguments. Let $B$ and $C$ be finite flat $A$-algebras. Let $Z=(z_1,\ldots,z_n)$ and $W=(w_1,\ldots,w_{n'})$ be systems of finite generators of the $A$-algebras $B$ and $C$, respectively. Let $\psi:B\to C$ be an $A$-algebra homomorphism satisfying $\psi(Z)\subseteq W$. Put $X=(X_1,\ldots,X_n)$ and $Y=(Y_1,\ldots,Y_{n'})$ as before. Choose a ring homomorphism $\Psi:\okey[X]\to \okey[Y]$ which makes the following diagram commutative:
\[
\xymatrix{
\okey[X] \ar[r]^{\Psi}\ar[d] & \okey[Y]\ar[d]\\
B\ar[r]_{\psi} & C,
}
\]
where the vertical arrows are the maps $X_i\mapsto z_i$ and $Y_i\mapsto w_i$ defined using $\iota:\okey\to A$. Put $\Psi(X)=(\Psi(X_1),\ldots,\Psi(X_n))\in \okey[Y]^n$.

\begin{lem}\label{func}
The map $\Psi$ induces a morphism of $K$-affinoid varieties
\[
\Psi^*: X^j_K(C,W)\to X^j_K(B,Z).
\]
Moreover, the map induced on the set of geometric connected components
\[
\Psi^*: \pi_0^\geom(X^j_K(C,W)) \to \pi_0^\geom(X^j_K(B,Z))
\]
is independent of the choice of $\Psi$.
\end{lem}
\begin{proof}
This is implicit in \cite[Section 2]{Ha}, and its proof is similar to \cite[Lemma 1.9]{AS2}. Let $\bar{f}=\{\bar{f}_1,\ldots,\bar{f}_r\}$ be as above and $f$ be its lift by the surjection $\iota$, as in Subsection \ref{subsectub}. Define $\bar{g}=\{\bar{g}_1,\ldots,\bar{g}_{r'}\}$ and $g$ similarly for $C$. Then the kernel of the surjection $\okey[X]\to B$ is the ideal $(\pi^m, f_1,\ldots,f_r)$ and similarly for $C$. The polynomial $f_i(\Psi(X))\in \okey[Y]$ is contained in the ideal $(\pi^m,g_1,\ldots, g_{r'})$. This shows the implication
\[
|g_i(y)|\leq |\pi|^j \text{ for any }i\Rightarrow |f_i(\Psi(X))(y)|\leq |\pi|^j \text{ for any }i
\]
for any $y=(y_1,\ldots,y_{n'})\in \cO^{n'}_{\Kbar}$ and the first assertion follows. Moreover, we have a natural cartesian diagram
\[
\xymatrix{
X^j_K(B,Z)(\Kbar) \ar[r]\ar[d] & \Hom_{\okey\text{-alg.}}(\okey[X],\okbar)\ar[d]\\
\Hom_{A\text{-alg.}}(B,\okbar/\mathfrak{m}_{\Kbar}^j)\ar[r] & \Hom_{\okey\text{-alg.}}(\okey[X],\okbar/\mathfrak{m}_{\Kbar}^j),
}
\]
where $\mathfrak{m}^j_{\Kbar}=\{x\in \okbar\mid |x|\leq |\pi|^j\}$ and the vertical arrows are surjections. Since the fiber of the left vertical arrow is the polydisc of radius $|\pi|^j$ and it is connected, we have a commutative diagram
\[
\xymatrix{
\Hom_{A\text{-alg.}}(C,\okbar/\mathfrak{m}_{\Kbar}^j)\ar[r]^{\psi^*}\ar[d] & \Hom_{A\text{-alg.}}(B,\okbar/\mathfrak{m}_{\Kbar}^j)\ar[d]\\
\pi_0^\geom(X^j_K(C,W))\ar[r]_{\Psi^*} & \pi_0^\geom(X^j_K(B,Z))
}
\]
whose vertical arrows are surjections. The second assertion follows from this diagram.
\end{proof}

Applying Lemma \ref{func} to the case of $\id:B\to B$, we obtain the inverse system
\[
(\pi_0^\geom(X^j_K(B,Z)))_Z.
\]
\begin{lem}\label{const}
This inverse system is constant.
\end{lem}
\begin{proof}
This is also implicit in \cite{Ha}, and the proof is similar to \cite[Lemma 3.1]{AS1}. We may assume that $Z=(z_1,\ldots,z_n)$ and $W=(z_1,\ldots,z_n,z_{n+1})$. Let $f=\{f_1,\ldots, f_r\}$ be as before. Consider a system of finite generators of the kernel of the surjection $A[X_1,\ldots,X_{n+1}]\to B$ associated to $W$. Then its lift $g$ by $\iota$ can be taken as $g=\{f_1,\ldots, f_r,X_{n+1}-h\}$ with some $h\in \okey[X_1,\ldots,X_n]$. Thus the map 
\[
(x_1,\ldots,x_n,x_{n+1})\mapsto (x_1,\ldots,x_n,x_{n+1}-h(x_1,\ldots,x_n))
\]
induces an isomorphism of $K$-affinoid varieties $X^j_K(B,W)\to X^j_K(B,Z)\times D_K(0,|\pi|^j)$ fitting into the commutative diagram
\[
\xymatrix{
X^j_K(B,W)\ar[r]\ar[rd]& X^j_K(B,Z)\times D_K(0,|\pi|^j)\ar[d]^{\prjt_1} \\
& X^j_K(B,Z),
}
\]
where $D_K(0,|\pi|^j)$ is the one-dimensional disc of radius $|\pi|^j$ centered at the origin. This implies the lemma. 
\end{proof}

By Lemma \ref{func} and Lemma \ref{const},
\[
\cF^j_{(K,\iota)}(B)=\varprojlim_Z \pi_0^\geom(X^j_K(B,Z))
\]
defines a contravariant functor from the category of finite flat $A$-algebras to that of finite $G_K$-sets.



\subsection{Ramification correspondence between different characteristics}\label{subsubseccomp}

Now we assume $pA=0$ and fix a $k$-algebra structure of $A$ which gives a section of the reduction map $A\to k$, as in Subsection \ref{subsecimperf}. Let $(K,\iota)$ and $(F,\boldsymbol{\iota})$ be lifts of $A$ as in Subsection \ref{subseclift}. 

Let $B$ be a finite flat $A$-algebra, $Z$ be its system of finite generators and $j$ be a positive rational number satisfying $j\leq m$. Let $\bar{f}$ be a finite subset of $A[X]$ with respect to $Z$, as before. Let $\Kbar$, $\bC$, $\fR$, $K'$, $K'_\infty$ and $F'$ be as in Subsection \ref{subsecimperf}. From the definition of the $j$-th tubular neighborhoods $X_K^j(B,Z)$ and $X_F^j(B,Z)$, Theorem \ref{comparisonimperf} yields an isomorphism of finite $G_{K'_\infty}$-sets
\[
\rho^{K,F}_{\bar{f},\sharp Z}: \pi_0^\geom(X_K^j(B,Z))\to \pi_0^\geom(X_F^j(B,Z)),
\]
which is also denoted by $\rho^{K,F}_{B,Z}$. Then the main result of this subsection is the following.

\begin{prop}\label{compfunc}
The isomorphism $\rho^{K,F}_{B,Z}$ induces a natural isomorphism of functors
\[
\rho^{K,F}: \cF^j_{(K,\iota)}(\cdot)|_{G_{K'_\infty}} \to \cF^j_{(F,\boldsymbol{\iota})}(\cdot)|_{G_{F'}}
\]
from the category of finite flat $A$-algebras to that of finite $G_{K'_\infty}$-sets, via the isomorphism
\[
G_{K'_\infty}\simeq G_{F'}
\]
of the classical theory of fields of norms.
\end{prop}
\begin{proof}
Let $A$, $B$, $Z$ and $\bar{f}$ be as above. Recall that we considered a map $A\to A'$ of truncated discrete valuation rings in Subsection \ref{subsecimperf}. Put $B'=B\otimes_A A'$, which is a finite flat $A'$-algebra. Then $Z$ defines a system of finite generators of the $A'$-algebra $B'$, which is denoted by $Z'$. The kernel of the surjection $A'[X]\to B'$ associated to $Z'$ is generated by the image $\bar{f}'$ of $\bar{f}$ by the map $A\to A'$. Thus, from the definition of the map $\rho^{K,F}_{B,Z}$, it is enough to show a similar statement for the map $\rho^{K',F'}_{B',Z'}$. Namely, we may assume that the residue field $k$ is perfect.

Let $C$ be a finite flat $A$-algebra with system of finite generators $W$. Let $\psi:B\to C$ be an $A$-algebra homomorphism satisfying $\psi(Z)\subseteq W$. Put $n'=\sharp W$ and $Y=(Y_1,\ldots,Y_{n'})$. We choose a lift $\Psi:\okey[X]\to \okey[Y]$ of $\psi$ along the lift $(K,\iota)$ as in Subsection \ref{Hatt} and put $\Psi_i=\Psi(X_i)\in \okey[Y]$. 

Consider the adic spaces
\[
X^\ad_{\bC,0}=\Spa(\bC\la X\ra, \oc\la X\ra),\ Y^\ad_{\bC,0}=\Spa(\bC\la Y\ra, \oc\la Y\ra).
\]
Let $X^\ad_{\bC,\infty}$ be the adic space over $X^\ad_{\bC,0}$ considered in Subsection \ref{subseccomp}, and $Y^\ad_{\bC,\infty}$ be a similar adic space for $Y^\ad_{\bC,0}$. Let $X^{j,\ad}_{\bC}(B,Z)$ be the adic space associated to the base change of the rigid analytic space $X^j_K(B,Z)$ to $\Spv(\bC)$, as before. Let $X^{j,\ad}_{\bC,\infty}(B,Z)$ be its inverse image by the natural projection $p_{\infty,0}: X^\ad_{\bC,\infty}\to X^\ad_{\bC,0}$. We define adic spaces $Y^{j,\ad}_{\bC}(C,W)$ and $Y^{j,\ad}_{\bC,\infty}(C,W)$ similarly, using $Y^\ad_{\bC,0}$ and $Y^\ad_{\bC,\infty}$. The map $\Psi$ induces a morphism of adic spaces $\Psi^*:Y^\ad_{\bC,0}\to X^\ad_{\bC,0}$ which maps the rational subset $Y^{j,\ad}_{\bC}(C,W)$ to $X^{j,\ad}_{\bC}(B,Z)$. Thus we have a diagram of finite sets
\[
\xymatrix{
\pi_0(Y^{j,\ad}_{\bC,\infty}(C,W))\ar@{.>}[r]\ar[d] &  \pi_0(X^{j,\ad}_{\bC,\infty}(B,Z))\ar[d]\\
\pi_0(Y^{j,\ad}_{\bC}(C,W))\ar[r]_{\Psi^*} &  \pi_0(X^{j,\ad}_{\bC}(B,Z)),
}
\]
where the lower horizontal arrow is compatible with the Galois action. Since the vertical arrows are also bijections compatible with the Galois action by Lemma \ref{pi0toinf}, there exists a unique map
\[
\pi_0(\Psi^*)_{\infty}:\pi_0(Y^{j,\ad}_{\bC,\infty}(C,W))\to \pi_0(X^{j,\ad}_{\bC,\infty}(B,Z))
\]
which makes the above diagram commutative. From the definition, we see that this map is also compatible with the Galois action.

On the other hand, let 
\[
X^\ad_{\fR,0},\ Y^\ad_{\fR,0},\ X^\ad_{\fR,\infty},\ Y^\ad_{\fR,\infty},\ X^{j,\ad}_{\fR,\infty}(B,Z)\text{ and }Y^{j,\ad}_{\fR,\infty}(C,W)
\]
be similar adic spaces on the side of $F$. We choose $\Psi^\flat_i\in \oef[Y]$ so that the images of $\Psi_i$ and $\Psi^\flat_i$ in the ring $A[Y]$ by the surjections induced by $\iota$ and $\boldsymbol{\iota}$ coincide with each other. Let $\Psi^\flat:\oef[X]\to \oef[Y]$ be the map defined by $X_i\mapsto \Psi^\flat_i$. Then it is a lift of $\psi$ along $(F,\boldsymbol{\iota})$ as in Subsection \ref{Hatt}. This induces a morphism of adic spaces $(\Psi^\flat)^*:Y^\ad_{\fR,0}\to X^\ad_{\fR,0}$. Note that, by the choice of $\Psi^\flat_i$, Lemma \ref{sharp} yields the congruence
\begin{equation}\label{Psicongruence}
(\Psi^\flat_i)^\sharp \equiv \Psi_i\bmod \pi^m
\end{equation}
in the ring $\oc\la Y^{1/p^\infty}\ra$.

Since the integral domain $\oR\la Y^{1/p^\infty} \ra$ is perfect, we have the unique $p^l$-th root $(\Psi^\flat_i)^{1/p^l}$ of $\Psi^\flat_i$ in this ring. The map $X^{1/p^l}_i\mapsto (\Psi^\flat_i)^{1/p^l}$ defines a morphism of adic spaces $(\Psi^\flat_\infty)^*:Y^\ad_{\fR,\infty}\to X^\ad_{\fR,\infty}$ which fits into the commutative diagram
\[
\xymatrix{
Y^\ad_{\fR,\infty}\ar[r]^{(\Psi^\flat_\infty)^*}\ar[d] & X^\ad_{\fR,\infty}\ar[d]\\
Y^\ad_{\fR,0}\ar[r]_{(\Psi^\flat)^*}& X^\ad_{\fR,0}
}
\]
and induces a continuous map $(\Psi^\flat_\infty)^*: Y^{j,\ad}_{\fR,\infty}(C,W)\to X^{j,\ad}_{\fR,\infty}(B,Z)$. Moreover, we can see as above that the induced map
\[
\pi_0(Y^{j,\ad}_{\fR,\infty}(C,W))\overset{(\Psi^\flat_\infty)^*}{\to} \pi_0(X^{j,\ad}_{\fR,\infty}(B,Z))
\]
is also compatible with the Galois action. Hence we are reduced to showing that the lower square of the diagram
\[
\xymatrix{
\pi_0(Y^{j,\ad}_{\bC}(C,W))\ar[r]^{\Psi^*} & \pi_0(X^{j,\ad}_{\bC}(B,Z)) \\
\pi_0(Y^{j,\ad}_{\bC,\infty}(C,W))\ar[r]^{\pi_0(\Psi^*)_\infty}\ar[d]_{\tau}\ar[u]^{\wr} & \pi_0(X^{j,\ad}_{\bC,\infty}(B,Z))\ar[d]^{\tau}\ar[u]_{\wr} \\
\pi_0(Y^{j,\ad}_{\fR,\infty}(C,W))\ar[r]_{(\Psi^\flat_\infty)^*} &\pi_0(X^{j,\ad}_{\fR,\infty}(B,Z))
}
\]
is commutative. 

This can be shown as in the proof of Lemma \ref{tauconn}: Take a point $y$ in a connected component of $Y^{j,\ad}_{\bC,\infty}(C,W)$ defined by $Y^{1/p^l}_i\mapsto y_{i,l}$ with some $y_{i,l}\in \oc$ satisfying $y_{i,l+1}^p=y_{i,l}$ for any $l$. Put $y_0=(y_{1,0},\ldots,y_{n',0})$. Note that $y_0\in Y^{j,\ad}_{\bC}(C,W)$ and $\Psi^*(y_0)\in X^{j,\ad}_{\bC}(B,Z)$. Moreover, the latter point is defined by the map $X_i\mapsto \Psi_i(y_0)$. Choose a system $(\Psi_i(y_0)^{1/p^l})_{l\geq 0}$ of its $p$-power roots in $\oc$ satisfying $(\Psi_i(y_0)^{1/p^{l+1}})^p=\Psi_i(y_0)^{1/p^l}.$ Then the map $X_i^{1/p^l}\mapsto \Psi_i(y_0)^{1/p^l}$ gives a point $\Psi^*(y_0)^{1/p^\infty}$ of the adic space $X^{j,\ad}_{\bC,\infty}(B,Z)$. From the definition, we see that the map $\pi_0(\Psi^*)_\infty$ sends the connected component containing $y$ to the connected component containing $\Psi^*(y_0)^{1/p^\infty}$.

It is enough to show that the points
\[
(\Psi^\flat_\infty)^*(\tau(y))\text{ and } \tau(\Psi^*(y_0)^{1/p^\infty})
\]
are contained in the same connected rational subset of the adic space $X^{j,\ad}_{\fR,\infty}(B,Z)$. Put $\uy_i=(y_{i,l})_l\in \oR$ and $\uy=(\uy_1,\ldots,\uy_{n'})$ as before. Then we have $\uy^\sharp=(y_{1,0},\ldots, y_{n',0})$ and the commutative diagram (\ref{iotadiagram}) implies $\uy\in Y^{j,\ad}_{\fR}(C,W)$. We let $\uy$ also denote the unique inverse image of this point in $Y^{j,\ad}_{\fR,\infty}(C,W)$. Let $V$ be the rational subset of $X^{\ad}_{\fR,\infty}$ defined by
\[
V=\{ z \in X^{\ad}_{\fR,\infty} \mid |(X_i-\Psi^\flat_i(\uy))(z)|\leq |u(z)|^m\}.
\]
From Lemma \ref{phomeo}, we see that $V$ is connected. By the definition of the map $(\Psi^\flat_\infty)^*$, the point $(\Psi^\flat_\infty)^*(\uy)$ lies in $V$ and the assumption $j\leq m$ implies $V\subseteq X^{j,\ad}_{\fR,\infty}(B,Z)$. By the commutative diagram (\ref{iotadiagram}), Lemma \ref{sharp} and the congruence (\ref{Psicongruence}), we obtain the equivalences
\begin{align*}
&|(\Psi^\flat_i-\Psi^\flat_i(\uy))^\sharp(y)|\leq |u^\sharp(y)|^m \Leftrightarrow |(\Psi^\flat_i)^\sharp(y)-(\Psi^\flat_i(\uy))^\sharp|\leq |\pi|^m\\
&\Leftrightarrow |\Psi_i(y_0)-\Psi_i(y_0)|\leq |\pi|^m \Leftrightarrow |\Psi_i(y_0)-(\Psi^\flat_i(\uy))^\sharp|\leq |\pi|^m\\
&\Leftrightarrow |(X_i-(\Psi^\flat_i(\uy))^\sharp)(\Psi^*(y_0)^{1/p^\infty})|\leq |\pi(\Psi^*(y_0)^{1/p^\infty})|^m\\
&\Leftrightarrow |(X_i-\Psi^\flat_i(\uy))^\sharp(\Psi^*(y_0)^{1/p^\infty})|\leq |u^\sharp(\Psi^*(y_0)^{1/p^\infty})|^m.
\end{align*} 
This implies the claim and concludes the proof of the proposition.
\end{proof}

\begin{cor}\label{FuncIndep}
Let $j$ be a positive rational number satisfying $j\leq m$. Let $(K_1,\iota_1)$ and $(K_2,\iota_2)$ be lifts of $A$. Then there exists a natural isomorphism
\[
\rho_{K_1,K_2}:\cF^j_{(K_1,\iota_1)} \to \cF^j_{(K_2,\iota_2)}
\]
of functors from the category of finite flat $A$-algebras to that of finite sets.
\end{cor}
\begin{proof}
Note that the functor $\cF^j_{(F,\boldsymbol{\iota})}$ is independent of the choice of an algebraic closure of $F$, up to a natural isomorphism. The corollary follows from this fact and Proposition \ref{compfunc}.
\end{proof}

\begin{rmk}
The natural isomorphism $\rho_{K_1,K_2}$ depends on various choices: it depends on the choices of a $k$-algebra structure and a uniformizer of $A$, a $p$-basis and a Cohen ring $C(k)$ of $k$, an algebraic closure, a local map $C(k)\to \cO_{K_1}$, a uniformizer, a lift of the fixed $p$-basis of $k$ and their compatible $p$-power roots for $K_1$ and similar choices for $K_2$. 
\end{rmk}



\subsection{Ramification of complete intersection $A$-algebras}

Let $A$ be a truncated discrete valuation ring of length $m$. Let $B$ be a finite flat $A$-algebra which is relatively of complete intersection (\cite[D\'{e}finition (19.3.6)]{EGA4-4}). The following lemma gives typical examples of such an extension $B/A$.

\begin{lem}\label{CI}
\begin{enumerate}
\item\label{CI-cdvr}
Let $K$ be a complete discrete valuation field with uniformizer $\pi$. Let $L$ be a finite extension of $K$ and $n$ be a positive integer. Then $\oel$ ({\it resp.} $\oel/(\pi^n)$) is a finite flat $\okey$-algebra ({\it resp.} $\okey/(\pi^n)$-algebra) which is relatively of complete intersection. 
\item\label{CI-tdvr}
Let $A$ and $B$ be truncated discrete valuation rings such that $B$ is a finite flat $A$-algebra. Then the $A$-algebra $B$ is relatively of complete intersection.
\end{enumerate}
\end{lem}
\begin{proof}
Since $\oel$ is a complete Noetherian regular local ring and $\pi$ is a regular element, the ring $\oel/(\pi)$ is a ring of complete intersection (\cite[D\'{e}finition (19.3.1)]{EGA4-4}). Then the first assertion follows from the definition and \cite[Corollaire (19.3.8)]{EGA4-4}. The second assertion also follows from the definition, since $B\otimes_A k$ is a truncated discrete valuation ring, and thus a ring of complete intersection.
\end{proof}

\begin{dfn}[\cite{HT}, Definition 3.2] 
Let $B$ be a finite flat $A$-algebra which is relatively of complete intersection, $(K,\iota)$ be a lift of $A$ and $j\in \bQ\cap (0,m]$. We say that the ramification of $B/A$ is bounded by $j$ if
\[
\sharp \cF^j_{(K,\iota)}(B)=\rank_A(B).
\]
\end{dfn}

This condition {\it{a priori}} depends on the choice of a lift $(K,\iota)$ of $A$. However, the following corollary shows that it is in fact independent of the choice of a lift, for the case of $pA=0$.

\begin{cor}\label{ObjIndep}
Suppose $pA=0$. Let $j$ be a positive rational number satisfying $j\leq m$. Let $B$ be a finite flat $A$-algebra relatively of complete intersection. Then the condition that the ramification of $B/A$ is bounded by $j$ is independent of the choice of a lift $(K,\iota)$ of $A$.
\end{cor}
\begin{proof}
This follows immediately from Corollary \ref{FuncIndep}.
\end{proof}

Next we study a relationship between ramification of finite flat $A$-algebras and of finite flat algebras over complete discrete valuation rings. Let $K$ be a complete discrete valuation field of residue characteristic $p>0$ with algebraic closure $\Kbar$. For any finite flat $\okey$-algebra $\tilde{B}$, put $\tilde{B}_K=\tilde{B}\otimes_{\okey}K$ and $\cF_K(\tilde{B})=\Hom_{K\text{-alg.}}(\tilde{B}_K,\Kbar)$. For any positive rational number $j$, let $\cF^j_K$ be the functor of (non-log) ramification theory for $K$ constructed in \cite[Subsection 3.1]{AS1}. It is a contravariant functor from the category of finite flat $\okey$-algebras to that of finite $G_K$-sets which is endowed with a natural map of finite $G_K$-sets
\[
\cF_K(\tilde{B})\to \cF^j_K(\tilde{B}).
\]
If the $\okey$-algebra $\tilde{B}$ is relatively of complete intersection, then this map is a surjection (\cite[Proposition 6.1]{AS1}). 

\begin{dfn}
For a finite flat $\okey$-algebra $\tilde{B}$ relatively of complete intersection, we say that the (non-log) ramification of $\tilde{B}/\okey$ is bounded by $j$ if
\[
\sharp \cF^j_K(\tilde{B})=\rank_{\okey}(\tilde{B}).
\]
For a finite extension $L/K$, we say that the ramification of $L/K$ is bounded by $j$ if the ramification of $\oel/\okey$ is bounded by $j$.
\end{dfn}
Note that, if the $K$-algebra $\tilde{B}_K$ is etale, then it is equivalent to the definition given in \cite[Definition 6.3]{AS1}. Moreover, if the ramification of $\tilde{B}/\okey$ is bounded by some $j$, then the $K$-algebra $\tilde{B}_K$ is etale. 

Let $A$ be a truncated discrete valuation ring of length $m$ and $(K,\iota)$ be a lift of $A$. For any finite flat $\okey$-algebra $\tilde{B}$, the $A$-algebra $B=\tilde{B}\otimes_{\okey,\iota}A$ is a finite flat $A$-algebra. If the $\okey$-algebra $\tilde{B}$ is relatively of complete intersection, then the $A$-algebra $B$ is also relatively of complete intersection. For any positive rational number $j\leq m$, \cite[Lemma 1]{Ha} yields a natural isomorphism of finite $G_K$-sets
\[
\cF^j_K(\tilde{B})\simeq \cF_{(K,\iota)}^j(B).
\]

\begin{lem}\label{CIlift}
Let $K$ be a complete discrete valuation field of residue characteristic $p>0$ with uniformizer $\pi$. Let $n$ be a positive integer and $j$ be a positive rational number satisfying $j\leq n$.
\begin{enumerate}
\item\label{CIlift-lift}{(\cite{HT}, Corollary 3.5)}
Let $\tilde{B}$ be a finite flat $\okey$-algebra which is relatively of complete intersection. Put $A=\okey/(\pi^n)$ and $B=\tilde{B}/(\pi^n)$. Then the ramification of $\tilde{B}/\okey$ is bounded by $j$ if and only if the ramification of $B/A$ is bounded by $j$.
 
\item\label{CIlift-sep}
Let $L$ be a finite extension of $K$. Put $A=\okey/(\pi^n)$ and $B=\oel/(\pi^n)$. If the ramification of $B/A$ is bounded by $j$, then $L$ is a separable extension of $K$.
\end{enumerate} 
\end{lem}
\begin{proof}
The first assertion follows from the definition and \cite[Lemma 1]{Ha}. Let us show the second assertion. By Lemma \ref{CI} (\ref{CI-cdvr}), the finite flat $\okey$-algebra $\oel$ is relatively of complete intersection. By \cite[Proposition 6.1]{AS1}, we have a surjection
\[
\cF_K(\oel)\to \cF^j_K(\oel).
\]
Hence we obtain the inequality
\[
[L:K]\geq \sharp \cF_K(\oel) \geq \sharp \cF^j_K(\oel)=\sharp \cF^j_{(K,\iota)}(B)=[L:K],
\]
which implies that the extension $L/K$ is separable. 
\end{proof}

\begin{cor}\label{LEnonlog}
Let $L_1/K_1$ and $L_2/K_2$ be extensions of complete discrete valuation fields of residue characteristic $p>0$. Let $\pi_{K_i}$ be a uniformizer of $K_i$. Let $e(K_i)$ be as in Section \ref{intro}. Let $m$ be a positive integer satisfying $m \leq \min_i e(K_i)$. Suppose that we have compatible isomorphisms of rings $\cO_{K_1}/(\pi_{K_1}^m)\simeq \cO_{K_2}/(\pi_{K_2}^m)$ and $\cO_{L_1}/(\pi_{K_1}^m)\simeq \cO_{L_2}/(\pi_{K_2}^m)$. Then, for any positive rational number $j\leq m$, the ramification of $L_1/K_1$ is bounded by $j$ if and only if the ramification of $L_2/K_2$ is bounded by $j$.
\end{cor}
\begin{proof}
By \cite[Lemma 1]{Ha} and Corollary \ref{FuncIndep}, we obtain the equalities
\[
\sharp \cF^j_{K_1}(\cO_{L_1})=\sharp \cF^j_{(K_1,\iota_1)}(\cO_{L_1}/(\pi_{K_1}^m))=\sharp \cF^j_{(K_2,\iota_2)}(\cO_{L_2}/(\pi_{K_2}^m))=\sharp \cF^j_{K_2}(\cO_{L_2}),
\]
from which the corollary follows.
\end{proof}



\subsection{An equivalence of categories}\label{subsecequivcats}

Let $A$ be a truncated discrete valuation ring of length $m$ with residue field $k$ of characteristic $p>0$.

\begin{dfn}{(\cite[Section 2]{HT})}
A truncated discrete valuation ring $B$ is said as a finite extension of $A$ if it is a finite flat $A$-algebra for $m\geq 2$, and if it is a field which is a finite extension of $k$ for $m=1$.
\end{dfn}

Note that for any finite extension $B/A$ of truncated discrete valuation rings, the $A$-algebra $B$ is relatively of complete intersection by Lemma \ref{CI} (\ref{CI-tdvr}).

Any finite extension $B/A$ of truncated discrete valuation rings has the following lifting property, which is shown in the first part of the proof of \cite[Proposition 2.2]{HT}. This proposition also states that $L/K$ can be taken to be finite separable. However, the proof of this latter part seems to have a gap, since it is not clear in general that we have the equality $\frp'\cap R=\frq$ with their notation.

\begin{lem}\label{liftB}
Let $B$ be a finite flat $A$-algebra which is relatively of complete intersection and $(K,\iota)$ be a lift of $A$. Then there exist a finite flat $\okey$-algebra $\tilde{B}$ which is relatively of complete intersection and a cartesian diagram
\[
\xymatrix{
\okey \ar[r]\ar[d]_{\iota} & \tilde{B}\ar[d] \\
A \ar[r] & B.
}
\]
Moreover, if $B/A$ is a finite extension of truncated discrete valuation rings, then the $\okey$-algebra $\tilde{B}$ can be taken to be the ring of integers $\oel$ of a finite extension $L/K$.
\end{lem}
\begin{proof}
We present a proof for the convenience of the reader. We may assume that $B$ is local with maximal ideal $\mathfrak{m}_B$ and residue field $k_B$. By assumption, the ring $B\otimes_A k$ is a ring of complete intersection. Fix an $A$-algebra surjection
\[
\bar{\bA}=A[X_1,\ldots,X_n]\to B.
\]
Consider the surjection $\bA=\okey[X_1,\ldots,X_n]\to \bar{\bA}$ induced by $\iota$ and let $\mathfrak{m}$ ({\it resp.} $\bar{\mathfrak{m}}$) be the maximal ideal of $\bA$ ({\it resp.} $\bar{\bA}$) which is the inverse image of $\mathfrak{m}_B$. The completions of the local rings $\bar{\bA}_{\bar{\mathfrak{m}}}$ and $\bA_{\mathfrak{m}}$ are denoted by $\bar{R}$ and $R$, respectively. Let $\mathfrak{m}_R$ be the maximal ideal of the local ring $R$ and $\pi$ be a uniformizer of $\okey$.

The local ring $\bar{R}$ is a flat $A$-algebra such that $\bar{R}\otimes_A k$ is regular. Then \cite[Corollaire (19.3.5)]{EGA4-4} implies that the kernel $\bar{\mathfrak{n}}$ of the surjection $\bar{R}\to B$ is generated by a regular sequence $(\bar{g}_1,\ldots, \bar{g}_r)$ of $\bar{R}$. Let $g_i$ be a lift of $\bar{g}_i$ in $R$. Since the sequence $(\pi^m,g_1,\ldots,g_r)$ in the ideal $\mathfrak{m}_R$ is regular and the local ring $R$ is Noetherian, the sequence $(g_1,\ldots, g_r,\pi^m)$ is also regular by \cite[Corollaire (15.1.11)]{EGA4-1}. Then the $\okey$-algebra $\tilde{B}=R/(g_1,\ldots, g_r)$ is flat and $\pi$-adically complete. From the definition, we have $\tilde{B}/(\pi^m)\simeq B$ and thus the $\okey$-algebra $\tilde{B}$ is finite. By assumption, the $k$-algebra $\tilde{B}\otimes_{\okey}k$ is relatively of complete intersection. From \cite[Corollaire (19.3.8)]{EGA4-4}, we see that the $\okey$-algebra $\tilde{B}$ is also relatively of complete intersection.

Next we assume that $B/A$ is a finite extension of truncated discrete valuation rings. If $m=1$, then we can construct $L$ as in the lemma by taking an unramified extension of $K$ for the case where $B/k$ is separable, an extension generated by a lift of a generator of $B/k$ for the case where $B/k$ is purely inseparable of degree $p$ and by induction for the general case.

Let us consider the case of $m\geq 2$. We identify the residue field of $R$ with $k_B$. Put $\mathfrak{n}=\Ker(R\to B)=(\pi^m,g_1,\ldots,g_r)$. Since $m\geq 2$, the maximal ideal $\mathfrak{m}_B$ is not zero and we have an exact sequence of $k_B$-vector spaces
\[
0 \to \mathfrak{n}/(\mathfrak{n}\cap \mathfrak{m}_R^2)\to \mathfrak{m}_R/\mathfrak{m}_R^2 \to \mathfrak{m}_B/\mathfrak{m}_B^2\to 0,
\]
where the $k_B$-vector space on the middle ({\it resp.} right) term is of dimension $n+1$ ({\it resp.} one). The assumption $m\geq 2$ also implies that the $k_B$-vector space on the left term is generated by the images of $g_1,\ldots, g_r$, and thus $r\geq n$. Since the local ring $R$ is Cohen-Macaulay of dimension $n+1$, the maximal length of regular sequences in $\mathfrak{m}_R$ is $n+1$. Hence we obtain $r=n$ and the images of $g_1,\ldots, g_n$ in  $\mathfrak{n}/(\mathfrak{n}\cap \mathfrak{m}_R^2)$ are linearly independent over $k_B$. This means that $g_1,\ldots, g_n$ form a part of a system of regular parameters of the regular local ring $R$. Thus the local ring $\tilde{B}=R/(g_1,\ldots,g_n)$ is regular of dimension one, namely a discrete valuation ring. Since it is flat over $\okey$, the map $\okey\to\tilde{B}$ is an injection. Since it is $\pi$-adically complete, it is also a complete discrete valuation ring and the second assertion follows.
\end{proof}

Let $B/A$ and $B'/A$ be finite extensions of truncated discrete valuation rings. For any lift $(K,\iota)$ of $A$ and any positive rational number $j\leq m$, we define an equivalence relation $\sim_j$ on the set $\Hom_{A\text{-alg.}}(B,B')$ by
\[
\psi\sim_j \psi'\Leftrightarrow \cF^j_{(K,\iota)}(\psi)=\cF^j_{(K,\iota)}(\psi')
\]
for any $A$-algebra homomorphisms $\psi,\psi': B\to B'$. 

\begin{lem}\label{MorIndep}
The equivalence relation $\sim_j$ is independent of the choice of a lift $(K,\iota)$. Moreover, it is compatible with the composite.
\end{lem}
\begin{proof}
Let $(K',\iota')$ be another lift of $A$. By Corollary \ref{FuncIndep}, we have a commutative diagram
\[
\xymatrix{
\Hom_{A\text{-alg.}}(B,B')\ar[r]\ar[rd] & \Map(\cF^j_{(K,\iota)}(B'), \cF^j_{(K,\iota)}(B))\ar[d]^{\wr}\\
& \Map(\cF^j_{(K',\iota')}(B'), \cF^j_{(K',\iota')}(B)),
}
\]
where the vertical arrow is a bijection. This implies the first assertion. If $\psi_1\sim_j \psi'_1$ and $\psi_2\sim_j \psi'_2$, then
\begin{align*}
\cF^j_{(K,\iota)}(\psi_1\circ\psi_2)&=\cF^j_{(K,\iota)}(\psi_2)\circ\cF^j_{(K,\iota)}(\psi_1)\\
&=\cF^j_{(K,\iota)}(\psi'_2)\circ\cF^j_{(K,\iota)}(\psi'_1)=\cF^j_{(K,\iota)}(\psi'_1\circ\psi'_2)
\end{align*}
and the second assertion also follows.
\end{proof}

\begin{dfn}[\cite{HT}, Definition 3.3]\label{defFFP}
We define a category $\mathrm{FFP}^{\leqslant j}_A$ as follows: the object of $\mathrm{FFP}^{\leqslant j}_A$ is any finite extension $B/A$ of truncated discrete valuation rings such that the ramification of $B/A$ is bounded by $j$. The morphism of $\mathrm{FFP}^{\leqslant j}_A$ is defined by 
\[
\Hom_{\mathrm{FFP}^{\leqslant j}_A}(B,B')=\Hom_{A\text{-alg.}}(B,B')/\sim_j.
\]
\end{dfn}

\begin{thm}\label{CatIndep}
The category $\mathrm{FFP}^{\leqslant j}_A$ is independent of the choice of a lift $(K,\iota)$ of $A$.
\end{thm}
\begin{proof}
This follows from Corollary \ref{ObjIndep} and Lemma \ref{MorIndep}.
\end{proof}

On the other hand, let $(K,\iota)$ be a lift of $A$ and $\mathrm{FE}_K^{\leqslant j}$ be the category of finite separable extensions $L/K$ such that its ramification is bounded by $j$. For $j\leq m$, we have a natural covariant functor $r_{(K,\iota)}^{\leqslant j}: \mathrm{FE}_K^{\leqslant j}\to \mathrm{FFP}^{\leqslant j}_A$ defined by $L\mapsto \oel/(\pi^m)$, where $\pi$ is a uniformizer of $K$ and we consider the ring $\oel/(\pi^m)$ as an $A$-algebra by the isomorphism $\okey/(\pi^m)\simeq A$ induced by $\iota$. Indeed, this functor is well-defined by Lemma \ref{CIlift} (\ref{CIlift-lift}) for $m \geq 2$ and by \cite[Proposition 6.9]{AS1} for $m=1$.
Then the following proposition is shown in \cite{HT}, which we present a proof for the convenience of the reader.

\begin{prop}[\cite{HT}, Corollary 1.2]\label{HTmain}
The functor $r_{(K,\iota)}^{\leqslant j}:\mathrm{FE}_K^{\leqslant j}\to \mathrm{FFP}^{\leqslant j}_A$ is an equivalence of categories.
\end{prop}
\begin{proof}
First we show that the functor is essentially surjective. Let $B$ be an object of the category $\mathrm{FFP}^{\leqslant j}_A$. By Lemma \ref{liftB}, there exists a finite extension $L/K$ satisfying $\oel \otimes_{\okey, \iota}A\simeq B$. By Lemma \ref{CIlift}, the extension $L/K$ is separable and its ramification is bounded by $j$.

For the full faithfulness, let $L$ and $L'$ be objects of the category $\mathrm{FE}_K^{\leqslant j}$. Put $B=\oel/(\pi^m)$ and $B'=\cO_{L'}/(\pi^m)$. Then we have a commutative diagram
\[
\xymatrix{
\Hom_{K\text{-alg.}}(L,L')\ar[d]_{\wr} &  \\
\Hom_{\okey\text{-alg.}}(\oel,\cO_{L'}) \ar[r]\ar[d]_{\wr} & \Hom_{A\text{-alg.}}(B,B')\ar[d] \\
\Map_{G_K}(\cF^j_K(\cO_{L'}),\cF^j_K(\oel)) \ar[r]_-{\sim} & \Map_{G_K}(\cF^j_{(K,\iota)}(B'),\cF^j_{(K,\iota)}(B)).
}
\]
Here $\Map_{G_K}$ means the set of morphisms of finite $G_K$-sets. Note that the left vertical arrows are bijections, since the ramifications of $L/K$ and $L'/K$ are bounded by $j$. Moreover, the lower horizontal arrow is also a bijection by \cite[Lemma 1]{Ha}. The full faithfulness follows from this diagram and the definition of the equivalence relation $\sim_j$.
\end{proof}

\begin{cor}\label{equivcats}
Let $K_1$ and $K_2$ be complete discrete valuation fields with residue fields $k_1$ and $k_2$ of characteristic $p>0$, respectively. Let $e(K_i)$ be as in Section \ref{intro}. Let $j$ be a positive rational number satisfying $j\leq \min_i e(K_i)$. Suppose that the fields $k_1$ and $k_2$ are isomorphic to each other. Then there exists an equivalence of categories
\[
\mathrm{FE}_{K_1}^{\leqslant j} \simeq \mathrm{FE}_{K_2}^{\leqslant j}.
\]
In particular, there exists an isomorphism of topological groups
\[
G_{K_1}/G_{K_1}^j\simeq G_{K_2}/G_{K_2}^j,
\]
where $G_{K_i}^j$ is the $j$-th (non-log) upper ramification subgroup of $G_{K_i}$ (\cite[Section 3]{AS1}).
\end{cor}
\begin{proof}
Put $m=\min_i e(K_i)$. Let $\pi_{K_i}$ be a uniformizer of $K_i$. Note that if $\ch(K_i)=0$, then the ring $\cO_{K_i}$ can be considered as a finite totally ramified extension of a Cohen ring of $k_i$. Thus the ring $A_i=\cO_{K_i}/(\pi_{K_i}^m)$ is non-canonically isomorphic to $k_i[u]/(u^m)$. Hence the categories $\mathrm{FFP}^{\leqslant j}_{A_1}$ and $\mathrm{FFP}^{\leqslant j}_{A_2}$ are equivalent, and the first assertion follows from Proposition \ref{HTmain}. The second assertion can be shown verbatim as the proof of \cite[(3.5.1)]{De}.
\end{proof}



\section{Log ramification}

In this section, we prove a variant of Proposition \ref{compfunc} for log ramification (\cite[Subsection 3.2]{AS1}). 

Let $K$ be a complete discrete valuation field with residue field $k$ of characteristic $p>0$ and uniformizer $\pi$. Let $\Kbar$ be an algebraic closure of $K$. Let $L$ be a finite extension of $K$ with residue field $k_L$, uniformizer $\pi_L$ and additive valuation $v_L$ which is normalized as $v_L(\pi_L)=1$. Let $e_{L/K}$ be the relative ramification index of the extension $L/K$. Let $Z=(z_1,\ldots,z_n)$ be a system of finite generators of the $\okey$-algebra $\oel$ and $P$ be a subset of $\{1,\ldots, n\}$ such that the set $\{z_i\mid i\in P\}$ contains a uniformizer of $\oel$ and does not contain the zero element. Such a pair $(Z,P)$ is referred to as a log system of generators of the $\okey$-algebra $\oel$. Put $e_i=v_L(z_i)$ for any $i\in P$. Consider the surjection $\okey[X]\to \oel$ associated to $Z$ and write its kernel as $(f_1,\ldots,f_r)$. For any $i\in P$, we choose a lift $g_i\in \okey[X]$ of the unit $z_i^{e_{L/K}}/\pi^{e_i}$ by this surjection. For any $i,i'\in P$, we also choose a lift $h_{i,i'}\in \okey[X]$ of the unit $z_{i'}^{e_i}/z_i^{e_{i'}}$. For any positive rational number $j$, the $j$-th log tubular neighborhood $X^j_{K,\log}(\oel,Z,P)$ of the $\okey$-algebra $\oel$ with respect to $(Z,P)$ is the $K$-affinoid variety defined as
\[
\left\{\begin{array}{l|l}
& |f_i(x)|\leq |\pi|^j\text{ for any }i,\\
x\in \Spv(K\la X\ra) & |(X_i^{e_{L/K}}-\pi^{e_i}g_i)(x)|\leq |\pi|^{j+e_i}\text{ for any }i\in P,\\
& |(X_{i'}^{e_i}-X_i^{e_{i'}}h_{i,i'})(x)|\leq |\pi|^{j+e_i e_{i'}/e_{L/K}}\text{ for any }i,i'\in P.
\end{array}\right\}
\]
Then the $K$-affinoid variety $X^j_{K,\log}(\oel,Z,P)$ is independent of the choice of $f_i$, $g_i$ and $h_{i,i'}$. Though $L/K$ is assumed to be separable in \cite[Subsection 3.2]{AS1}, a verbatim argument shows that the inverse system of finite $G_K$-sets
\[
\{\pi_0^\geom(X^j_{K,\log}(\oel,Z,P))\}_{(Z,P)}
\]
is constant also for the case where $L/K$ is not separable. This gives a contravariant functor
\[
\oel\mapsto \cF^j_{K,\log}(\oel)=\varprojlim_{(Z,P)}\pi_0^\geom(X^j_{K,\log}(\oel,Z,P))
\]
from the category of rings of integers of finite extensions of $K$ to that of finite $G_K$-sets. We have a natural map of finite $G_K$-sets
\[
\cF_K(\oel)\to \cF^j_{K,\log}(\oel),
\]
which is a surjection if the finite extension $L/K$ is separable (\cite[Proposition 9.3 (i)]{AS1}). 

\begin{lem}\label{sepinsep}
Let $L/K$ be a finite extension and $M$ be the separable closure of $K$ in $L$. Then we have
\[
\sharp \cF^j_{K,\log}(\oel)=\sharp \cF^j_{K,\log}(\oem).
\]
\end{lem}
\begin{proof}
We may assume $\ch(K)=p$. It suffices to show the equality in the lemma for any purely inseparable extension $L/M$ of degree $p$ of finite extensions of $K$. Let $k_M$ and $\pi_M$ be the residue field and a uniformizer of $M$. Let $z_1,\ldots,z_{n-1}$ be a lift in $\oem$ of a system of finite generators of the finite extension $k_M/k$ and put $z_n=\pi_M$. Then $Z=(z_1,\ldots,z_{n})$ and $P=\{n\}$ form a log system of generators of the $\okey$-algebra $\oem$. Let $(f_1,\ldots, f_r)$ be the kernel of the associated surjection $\okey[X]\to \oem$ and $g$ be a lift of the unit $\pi_M^{e_{M/K}}/\pi$ by this surjection. Then the log tubular neighborhood $X^j_{K,\log}(\oem,Z,P)$ is the affinoid subdomain 
\[
\{x\in \Spv(K\la X\ra) \mid
|f_i(x)|\leq |\pi|^j\text{ for any }i,\ |(X_n^{e_{M/K}}-\pi g)(x)|\leq |\pi|^{j+1}
\}
\]
of the $n$-dimensional unit polydisc $D^n_K$ over $K$. 

Suppose that $L/M$ is totally ramified. Then we have an isomorphism of $\oem$-algebras
\[
\oem[Y]/(Y^p-\pi_M)\simeq \oel.
\]
Let $z'_n$ be the image of $Y$ by this isomorphism, which is a uniformizer of $L$. Put $z'_i=z_i$ for $i\leq n-1$. Then $Z'=\{z'_1,\ldots,z'_{n-1},z'_n\}$ and $P'=\{n\}$ form a log system of generators of the $\okey$-algebra $\oel$. The kernel of the associated surjection $\okey[X]\to \oel$ is the ideal
\[
(f_1(X_1,\ldots, X_{n-1},X_n^p),\ldots, f_r(X_1,\ldots, X_{n-1},X_n^p)).
\]
The surjections associated to $Z$ and $Z'$ fit into the commutative diagram
\[
\xymatrix{
\okey[X]\ar[r]\ar[d] & \oem \ar[d]\\
\okey[X]\ar[r]& \oel,
}
\]
where the right vertical arrow is the natural inclusion and the left vertical arrow is defined by $X_i\mapsto X_i$ for $i\leq n-1$ and $X_n\mapsto X_n^p$. Since the natural inclusion sends $z_n^{e_{M/K}}/\pi$ to $(z'_n)^{e_{L/K}}/\pi$, the latter element is lifted to $g(X_1,\ldots,X_{n-1},X_n^p)$ by the surjection $\okey[X]\to \oel$. Thus the log tubular neighborhood $X^j_{K,\log}(\oel,Z',P')$ is the affinoid subdomain 
\[
\left\{\begin{array}{l|l}
x\in \Spv(K\la X\ra) &\begin{array}{l}
|f_i(X_1,\ldots,X_{n-1},X_n^p)(x)|\leq |\pi|^j\text{ for any }i,\\
|(X_n^{p e_{M/K}}-\pi g(X_1,\ldots, X_{n-1},X_n^p))(x)|\leq |\pi|^{j+1}
\end{array}
\end{array}\right\}
\]
of the $n$-dimensional unit polydisc $D^n_K$. Let $\bC$ be the completion of $\Kbar$. To compare the sets of geometric connected components, we may pass to an adic space over $\Spa(\bC,\oc)$ and consider the sets of connected components by Lemma \ref{connCp}. Let $X^{j,\ad}_{\bC,\log}(\oem,Z,P)$ be the adic space associated to the base change of $X^j_{K,\log}(\oem,Z,P)$ to $\Spv(\bC)$. We also define $X^{j,\ad}_{\bC,\log}(\oel,Z',P')$ and $D^{n,\ad}_\bC$ similarly. From the definition, we see that the rational subset $X^{j,\ad}_{\bC,\log}(\oel,Z',P')$ is the inverse image of the rational subset $X^{j,\ad}_{\bC,\log}(\oem,Z,P)$ of $D^{n,\ad}_\bC$ by the map
\[
D^{n,\ad}_\bC \to D^{n,\ad}_\bC,\ (X_1,\ldots,X_{n-1},X_n)\mapsto (X_1,\ldots,X_{n-1},X_n^p).
\]
Since this map is a homeomorphism, the claim follows in this case.

Next suppose that $L/M$ is of relative ramification index one. We have an isomorphism of $\oem$-algebras
\[
\oem[X_{n+1}]/(X^p_{n+1}-a)\to \oel,\ X_{n+1}\mapsto z'_{n+1}=a^{1/p}
\]
with some $a\in \oem^\times$. Put $z'_i=z_i$ for $i\leq n$. Then $Z'=(z'_1,\ldots,z'_{n+1})$ and $P=\{n\}$ form a log system of generators of the $\okey$-algebra $\oel$. Let $f'$ be a lift of the element $a$ by the surjection $\okey[X]\to \oem$ associated to $Z$. Define $X^{j,\ad}_{\bC,\log}(\oem,Z,P)$ and $X^{j,\ad}_{\bC,\log}(\oel,Z',P')$ similarly to the above case. Then $X^{j,\ad}_{\bC,\log}(\oel,Z',P')$ is the rational subset
\[
\left\{\begin{array}{l|l}
z\in D^{n+1,\ad}_\bC &\begin{array}{l}
|f_i(X_1,\ldots,X_n)(z)|\leq |\pi(z)|^j\text{ for any }i,\\
|(X^p_{n+1}-f'(X_1,\ldots,X_n))(z)|\leq |\pi(z)|^j,\\
|(X_n^{e_{M/K}}-\pi g(X_1,\ldots,X_n))(z)|\leq |\pi(z)|^{j+1}
\end{array}
\end{array}\right\}
\]
of the $(n+1)$-dimensional unit polydisc $D^{n+1,\ad}_\bC$. Now we consider the map
\[
\varphi: D^{n,\ad}_\bC \to D^{n,\ad}_\bC, \ (X_1,\ldots,X_{n})\mapsto (X_1^p,\ldots,X_{n}^p),
\]
which is also a homeomorphism. The inverse image of $X^{j,\ad}_{\bC,\log}(\oel,Z',P')\subseteq D^{n+1,\ad}_\bC$ by the map $\varphi\times\id$ is the rational subset
\[
\left\{\begin{array}{l|l}
z\in D^{n+1,\ad}_\bC &\begin{array}{l}
|f_i(X_1^p,\ldots,X^p_n)(z)|\leq |\pi(z)|^j\text{ for any }i,\\
|(X_{n+1}-f''(X_1,\ldots,X_n))(z)|\leq |\pi(z)|^{j/p},\\
|(X_n^{p e_{M/K}}-\pi g(X_1^p,\ldots,X_n^p))(z)|\leq |\pi(z)|^{j+1}
\end{array}
\end{array}\right\},
\]
where $f''$ is the element of the ring $\bC[X]$ satisfying $f''(X_1,\ldots,X_n)^p=f'(X^p_1,\ldots,X^p_n)$. Since this adic space is isomorphic to 
\[
\varphi^{-1}(X^{j,\ad}_{\bC,\log}(\oem,Z,P))\times_{\Spa(\bC,\oc)} D^{1,\ad}_\bC,
\]
the claim follows also in this case.
\end{proof}

\begin{dfn}
We say that the log ramification of a finite extension $L/K$ is bounded by $j$ if 
\[
\sharp \cF^j_{K,\log}(\oel)=[L:K].
\]
\end{dfn}

If $L/K$ is separable, it is equivalent to the definition given in \cite[Definition 9.4]{AS1}. Using Lemma \ref{sepinsep}, we can prove the following variant of Lemma \ref{CIlift} (\ref{CIlift-sep}) for log ramification.

\begin{lem}\label{CIliftlog}
Let $L/K$ be a finite extension and $j$ be a positive rational number. If the log ramification of $L/K$ is bounded by $j$, then the extension $L/K$ is separable.
\end{lem}
\begin{proof}
Let $M$ be the separable closure of $K$ in $L$. By Lemma \ref{sepinsep}, we have the inequality
\[
[L:K]\geq [M:K]=\sharp \cF_K(\oem)\geq \sharp \cF^j_{K,\log}(\oem)=\sharp \cF^j_{K,\log}(\oel)=[L:K],
\]
which implies $L=M$.
\end{proof}

Now let $L/K$ be a finite extension as before. Let $z_1,\ldots,z_{n-1}$ be elements of $\oel$ such that their images in $k_L$ generate the residue extension $k_L/k$. Put $z_n=\pi_L$. Then $Z=(z_1,\ldots,z_n)$ and $P=\{n\}$ form a log system of generators of the $\okey$-algebra $\oel$. Consider the associated surjection $\okey[X]\to \oel$ and write its kernel as $(f_1,\ldots, f_r)$. Let $g\in \okey[X]$ be a lift of the element $\pi_L^{e_{L/K}}/\pi$ by this surjection. Let $j$ be a positive rational number. Then the $j$-th log tubular neighborhood $X^{j}_{K,\log}(\oel,Z,P)$ is the affinoid variety
\[
\{x\in \Spv(K\la X\ra)\mid |f_i(x)|\leq |\pi|^j\text{ for any }i,\ |(X_n^{e_{L/K}}-\pi g)(x)|\leq |\pi|^{j+1}\},
\]
and we also have an isomorphism of finite $G_K$-sets
\[
\cF^j_{K,\log}(\oel)\simeq \pi_0^\geom(X^{j}_{K,\log}(\oel,Z,P)).
\]

Let $e=e(K)$ be as in Section \ref{intro}. Let $m$ be a positive integer satisfying $m\leq e-1$. Put $A=\okey/(\pi^m)$ and $A_+=\okey/(\pi^{m+1})$. Then the $A$-algebra $B=\oel/(\pi^m)$ is a truncated discrete valuation ring which is finite flat over $A$ and similarly for the $A_+$-algebra $B_+=\oel/(\pi^{m+1})$. The images of $\pi$ in $A$, $A_+$ and $\pi_L$ in $B_+$ are denoted by $\bar{\pi}$, $\bar{\pi}_+$ and $\bar{\pi}_{B_+}$, respectively. The natural map $\okey\to A_+$ gives a lift $(K,\iota_+)$ of $A_+$ and a lift $(K,\iota)$ of $A$.

On the other hand, by fixing a section $k\to A_+$ of the natural reduction map $A_+\to k$, we consider $A_+$ and $A$ as $k$-algebras. Put $F=k((u))$. The map $\oef=k[[u]]\to A_+$ sending $u$ to $\bar{\pi}_+$ gives a lift $(F,\boldsymbol{\iota}_+)$ of $A_+$ and a lift $(F,\boldsymbol{\iota})$ of $A$. 

Suppose that there exist a finite extension $E/F$ and an isomorphism of $A_+$-algebras $\oee\otimes_{\oef, \boldsymbol{\iota}_+}A_+\to B_+$. Note that we have $[L:K]=[E:F]$. Let $\pi_E$ be a uniformizer of $E$ lifting $\bar{\pi}_{B_+}$. Since the residue extensions of $L/K$ and $E/F$ are the same, we have the equality $e_{L/K}=e_{E/F}$.

\begin{lem}\label{loglift}
The images of the elements $\pi_L^{e_{L/K}}/\pi\in \oel$ and $\pi_E^{e_{E/F}}/u\in\oee$ in $B$ coincide with each other.
\end{lem}
\begin{proof}
From the definition, we see that the images of $\pi$ and $u$ in $A_+$ and the images of $\pi_L$ and $\pi_E$ in $B_+$ both coincide. Thus we have the equation in the ring $B_+$
\[
u(\pi_L^{e_{L/K}}/\pi)=\pi(\pi_L^{e_{L/K}}/\pi)=\pi_L^{e_{L/K}}=\pi_{E}^{e_{E/F}}=u(\pi_E^{e_{E/F}}/u).
\]
This implies the lemma.
\end{proof}

Let $\bar{z}_i$ be the image of $z_i$ in $B$. Then $\bar{Z}=(\bar{z}_1,\ldots,\bar{z}_n)$ gives a system of generators of the $A$-algebra $B$. Let $\bar{f}_i$ and $\bar{g}$ be the images of $f_i$ and $g$ in the ring $A[X]$. Let $\mathbf{f}_i$ and $\mathbf{g}$ be their lifts in $\oef[X]$ by the surjection $\boldsymbol{\iota}:\oef\to A$. Let $\mathbf{z}_i$ be a lift of $\bar{z}_i$ in $\oee$ satisfying $\mathbf{z}_n=\pi_E$. Then $\mathbf{Z}=(\mathbf{z}_1,\ldots,\mathbf{z}_n)$ and $\mathbf{P}=\{n\}$ form a log system of generators of the $\oef$-algebra $\oee$. Consider the associated surjection $\oef[X]\to \oee$, and let $\mathbf{g}'$ be a lift of the element $\pi_E^{e_{E/F}}/u$ by this surjection. Then the images $\mathbf{g}$ and $\mathbf{g}'$ in $B$ are both equal to the element of Lemma \ref{loglift}. Hence we have the congruence
\[
\mathbf{g}\equiv \mathbf{g}'\bmod (\mathbf{f}_1,\ldots,\mathbf{f}_r,u^m)
\]
in the ring $\oef[X]$. 

Let $j$ be a positive rational number satisfying $j\leq m-1$. The above congruence implies that the $j$-th log tubular neighborhood $X^{j}_{F,\log}(\oee,\mathbf{Z},\mathbf{P})$ of the $\oef$-algebra $\oee$ with respect to $(\mathbf{Z},\mathbf{P})$ is equal to the $F$-affinoid variety
\[
\{x\in \Spv(F\la X\ra)\mid |\mathbf{f}_i(x)|\leq |u|^j\text{ for any }i,\ |(X_n^{e_{E/F}}-u \mathbf{g})(x)|\leq |u|^{j+1}\}.
\]
Put $\bar{f}_{r+1}=X_{n}^{e_{L/K}}-\bar{\pi}\bar{g}$, $\bar{f}'=\{\bar{f}_1,\ldots,\bar{f}_r,\bar{f}_{r+1}\}$ and $j'=(j,\ldots,j,j+1)$. Then we have the equalities
\[
X^{j}_{K,\log}(\oel,Z,P)=X_K^{j'}(\bar{f}',n),\ X^{j}_{F,\log}(\oee,\mathbf{Z},\mathbf{P})=X_F^{j'}(\bar{f}',n).
\]
Hence Theorem \ref{comparisonimperf} implies the following theorem.
\begin{thm}\label{complog}
Let $L/K$ be a finite extension and $m$ be a positive integer satisfying $m\leq e(K)-1$. Put $A_+=\okey/(\pi^{m+1})$. Fix a section $k\to A_+$ of the reduction map and consider the $k$-algebra surjection $\boldsymbol{\iota}_+: k[[u]]\to A_+$ defined by $u\mapsto \pi$. Put $F=k((u))$. Let $E/F$ be a finite extension with an isomorphism of $A_+$-algebras \[
\oee\otimes_{\oef,\boldsymbol{\iota}_+}A_+\to \oel/(\pi^{m+1}).
\]
Let $j$ be a positive rational number satisfying $j\leq m-1$. Then there exists an isomorphism of finite $G_{K'_\infty}$-sets
\[
\cF_{K,\log}^{j}(\oel)\to \cF_{F,\log}^{j}(\oee)
\]
via the isomorphism $G_{K'_\infty}\simeq G_{F'}$ of the classical theory of fields of norms. 
\end{thm}
\begin{flushright}
\qed
\end{flushright}

\begin{cor}\label{LElog}
Let $L_1/K_1$ and $L_2/K_2$ be extensions of complete discrete valuation fields of residue characteristic $p>0$. Let $\pi_{K_i}$ be a uniformizer of $K_i$. Let $m$ be a positive integer satisfying $m \leq \min_i e(K_i)$. Suppose that we have compatible isomorphisms of rings $\cO_{K_1}/(\pi_{K_1}^m)\simeq \cO_{K_2}/(\pi_{K_2}^m)$ and $\cO_{L_1}/(\pi_{K_1}^m)\simeq \cO_{L_2}/(\pi_{K_2}^m)$. Then, for any positive rational number $j\leq m-2$, the log ramification of $L_1/K_1$ is bounded by $j$ if and only if the log ramification of $L_2/K_2$ is bounded by $j$.
\end{cor}
\begin{proof}
Put $A_+=\cO_{K_1}/(\pi_{K_1}^m)$ and $B_+=\cO_{L_1}/(\pi_{K_1}^m)$. We choose a uniformizer $\pi_{K_2}$ such that its image in $A_+$ coincides with the image of $\pi_{K_1}$. Fix a section $k\to A_+$ of the reduction map and consider the lift $(F,\boldsymbol{\iota}_+)$ of $A_+$ as above. Since we have $m>2$ by assumption, $B_+/A_+$ is a finite extension of truncated discrete valuation rings and Lemma \ref{liftB} enables us to find a finite extension $E/F$ with an isomorphism of $A_+$-algebras $\oee\otimes_{\oef, \boldsymbol{\iota}_+}A_+\to B_+$. Then Theorem \ref{complog} yields bijections
\[
\cF^j_{K_1,\log}(\cO_{L_1})\simeq \cF^j_{F,\log}(\cO_{E})\simeq \cF^j_{K_2,\log}(\cO_{L_2})
\]
for any positive rational number $j\leq m-2$. This concludes the proof.
\end{proof}

\begin{rmk}
Here we do not study any functoriality of the isomorphism of Theorem \ref{complog} similar to Proposition \ref{compfunc}, for the following reason: Let $L$ and $L'$ be finite separable extensions of $K$. Let $(Z,P)$ be a log system of generators of the $\okey$-algebra $\oel$ and $(Z',P')$ be that of $\cO_{L'}$. Consider an inclusion of $K$-algebras $\psi: L\to L'$ satisfying $\psi(Z)\subseteq Z'$ and $\psi(P)\subseteq P'$. Then $P'$ contains a uniformizer of $L$. However, this forces us to include in defining equations of the $j$-th log tubular neighborhood of the $\okey$-algebra $\cO_{L'}$ the following equation:
\[
|(X_i^{e_{L'/K}}-\pi^{e_{L'/L}}g)(x)|\leq |\pi|^{j+e_{L'/L}}.
\]
Since $e_{L'/L}$ can be arbitrarily large, we cannot connect affinoid varieties of different characteristics functorially using $\bmod\ \pi^m$ for a fixed $m$ as we did in the non-log case.
\end{rmk}



\section{Compatibility of Scholl's higher fields of norms with ramification}

In this section, we prove that Scholl's theory of higher fields of norms (\cite{Scholl}) is compatible with the ramification theory of Abbes-Saito. Let $d$ be a non-negative integer. Let $K_\bullet=(K_n)_{n\geq 0}$ be a strictly deeply ramified extension of $d$-big local fields of mixed characteristic $(0,p)$ (\cite[Subsection 1.3]{Scholl}). In particular, there exists a positive integer $n_0$ and an element $\xi\in \cO_{K_{n_0}}$ satisfying $|p|\leq |\xi|<1$ such that for any $n\geq n_0$, the relative ramification index $e_{K_{n+1}/K_n}$ is equal to $p$ and the $p$-th power Frobenius map induces a surjection $\cO_{K_{n+1}}/(\xi)\to \cO_{K_n}/(\xi)$.  Moreover, for any $n\geq n_0$, we can choose a uniformizer $\pi_{K_n}$ of $K_n$ satisfying $\pi_{K_{n+1}}^p\equiv \pi_{K_n}\bmod \xi$. Put $K_\infty=\cup_n K_n$ and
\[
X^{+}=X^+(K_\infty)=\varprojlim_{n\geq n_0,\Phi}\cO_{K_n}/(\xi),
\]
where all the transition maps are the $p$-th power Frobenius maps. Set $\Pi=(\pi_{K_n})_{n\geq n_0}$. Let $k_n$ be the residue field of $K_n$. Then $X^+$ is a complete discrete valuation ring of characteristic $p$ with uniformizer $\Pi$ and residue field
\[
k'=\varprojlim_{n\geq n_0,\Phi} k_n
\]
(\cite[Theorem 1.3.2]{Scholl}). Moreover, $X^+$ is independent of the choice of $n_0$ and $\xi$. Put $X=\Frac(X^+)$. 

Let $L_\infty$ be a finite extension of $K_\infty$. Then $L_\infty$ can be written as $L_\infty=K_\infty L_0$ with some finite extension $L_0/K_0$ inside $L_\infty$. Put $L_n=K_nL_0$. Then $L_\bullet=(L_n)_{n\geq 0}$ is also strictly deeply ramified for any $\xi'$ satisfying $|\xi|<|\xi'|<1$ with some $n'_0$, and we can define a complete discrete valuation ring $X^+(L_\infty)$ by a similar construction to $X^+$ for $L_\bullet$. Put $X(L_\infty)=\Frac(X^+(L_\infty))$. Then $L_\infty\mapsto X(L_\infty)$ defines an equivalence of categories from the category of finite extensions of $K_\infty$ to that of finite separable extensions of $X$ (\cite[Theorem 1.3.5]{Scholl}). In particular, for any finite Galois extension $L_\infty/K_\infty$, we have an isomorphism
\[
\Gal(L_\infty/K_\infty)\simeq \Gal(X(L_\infty)/X),
\]
which induces an isomorphism of absolute Galois groups $G_{K_\infty}\simeq G_X$.

\begin{dfn}
Let $L_\infty/K_\infty$ be a finite separable extension and $L_n$ be as above. We say that the ramification ({\it resp.} log ramification) of $L_\infty/K_\infty$ is bounded by $j$ if the ramification ({\it resp.} log ramification) of $L_n/K_n$ is bounded by $j$ for any sufficiently large $n$.
\end{dfn}

Then the main theorem of this section is the following, which reproves a result recently obtained by Shun Ohkubo using a totally different method.

\begin{thm}\label{Schollthm}
The ramification ({\it resp.} log ramification) of $L_\infty/K_\infty$ is bounded by $j$ if and only if the ramification ({\it resp.} log ramification) of $X(L_\infty)/X$ is bounded by $j$.
\end{thm}
\begin{proof}
We let $v_{K_n}$ denote the additive valuation of $K_n$ normalized as $v_{K_n}(\pi_{K_n})=1$ and $e(K_n)$ be the absolute ramification index of $K_n$. By replacing $\xi$ by $\xi'$, we may assume $n_0=n'_0$ and $\xi=\xi'$. Then we have $v_{K_n}(\xi)<e(K_n)$. Note that $v_{K_n}(\xi)$ can be arbitrarily large by increasing $n$. Take any positive integer $n\geq n_0$ satisfying $j\leq v_{K_n}(\xi)-2$. Put $m=v_{K_n}(\xi)$. Then the surjections 
\[
X^+\overset{\prjt_n}{\to} \cO_{K_n}/(\pi_{K_n}^m)\gets \cO_{K_n} 
\]
give two lifts of the truncated discrete valuation ring $\cO_{K_n}/(\pi_{K_n}^m)$ of length $m$, which is killed by $p$. Moreover, the diagram
\[
\xymatrix{
X^+(L_\infty) \ar[r]^{\prjt_n} & \cO_{L_n}/(\pi_{K_n}^m) & \cO_{L_n}\ar[l]\\
X^+ \ar[u]\ar[r]_-{\prjt_n} & \cO_{K_n}/(\pi_{K_n}^m) \ar[u] &  \cO_{K_n}\ar[l]\ar[u]
}
\]
is cartesian. Hence Corollary \ref{LEnonlog} and Corollary \ref{LElog} imply the theorem.
\end{proof}

For any positive rational number $j$, let $\mathrm{FE}^{\leqslant j}_{K_\infty}$ ({\it resp.} $\mathrm{FE}^{\leqslant j}_{K_\infty,\log}$) be the category of finite separable extensions $L_\infty/K_\infty$ whose ramification ({\it resp. log ramification}) is bounded by $j$. Put
\[
G_{K_\infty}^j=\bigcap_{L_\infty\in \mathrm{FE}^{\leqslant j}_{K_\infty}} G_{L_\infty},\ G_{K_\infty,\log}^j=\bigcap_{L_\infty\in \mathrm{FE}^{\leqslant j}_{K_\infty,\log}} G_{L_\infty}.
\]
Let $G_X^j$ and $G_{X,\log}^j$ be the $j$-th non-log and log upper ramification subgroups of $G_X$, respectively (\cite[Section 3]{AS1}). 
\begin{cor}
The isomorphism $G_{K_\infty}\simeq G_X$ induces isomorphisms
\[
G_{K_\infty}^j\simeq G_X^j,\ G_{K_\infty,\log}^j\simeq G_{X,\log}^j.
\]
\end{cor}
\begin{proof}
By Theorem \ref{Schollthm}, the equivalence of higher fields of norms induces an isomorphism
\[
\bigcap_{L_\infty\in \mathrm{FE}^{\leqslant j}_{K_\infty}} G_{L_\infty} \simeq \bigcap_{X'\in \mathrm{FE}^{\leqslant j}_{X}} G_{X'}
\]
and the latter group is equal to $G_X^j$. The assertion on log ramification follows similarly.
\end{proof}



\section{An application to a generalization of the Hasse-Arf theorem}\label{HasseArf}

Finally, we give an application of Theorem \ref{comparisonimperf} to a generalization of the Hasse-Arf theorem to the case of imperfect residue field, though it is under a very restrictive condition. 

Let $K$ be a complete discrete valuation field of residue characteristic $p>0$. Let $k$ be the residue field of $K$. Let $L/K$ be a finite separable extension. Then the Artin conductor $c(L/K)$ and the Swan conductor $c_{\log}(L/K)$ of the extension $L/K$ are defined as
\begin{align*}
c(L/K)&=\inf\{j\in \bQ_{>0}\mid \cF_K(\oel)\to \cF^j_{K}(\oel)\text{ is a bijection}\},\\
c_{\log}(L/K)&=\inf\{j\in \bQ_{>0}\mid \cF_K(\oel)\to \cF^j_{K,\log}(\oel)\text{ is a bijection}\},
\end{align*}
which are known to be non-negative rational numbers (\cite[Proposition 6.4 and Proposition 9.5]{AS1}). For their integrality, we have the following theorem of Xiao.

\begin{thm}\label{xiao}
Let $L/K$ be a finite abelian extension.
\begin{enumerate}
\item\label{xiaop}{(\cite{Xiao1}, Corollary 4.4.3)}
Suppose $\ch(K)=p$. Then $c(L/K)$ and $c_{\log}(L/K)$ are integers.
\item\label{xiao0}{(\cite{Xiao2}, Theorem)}
Suppose $\ch(K)=0$. 
\begin{enumerate}
\item $c(L/K)$ is an integer if $K$ is not absolutely unramified. 
\item $c_{\log}(L/K)$ is an integer if $p>2$, and $c_{\log}(L/K)\in \bZ[1/2]$ if $p=2$.
\end{enumerate}
\end{enumerate}
\end{thm}

By using Theorem \ref{comparisonimperf}, we can prove the following theorem on the integrality of the conductors, which includes some new cases for log ramification and $p=2$.

\begin{thm}\label{HA}
Suppose $\ch(K)=0$. Let $L/K$ be a finite abelian extension and $e$ be the absolute ramification index of $K$.
\begin{enumerate}
\item\label{HAnonlog}
If $c(L/K)< e$, then $c(L/K)$ is an integer.
\item\label{HAlog}
If $c_{\log}(L/K)< e-2$, then $c_{\log}(L/K)$ is an integer.
\end{enumerate}
\end{thm}
\begin{proof}
Let $\pi$ be a uniformizer of $K$. Put $m=n=e$ for the non-log case, and $m=e-1$, $n=e-2$ for the log case. We also put $A=\okey/(\pi^m)$ and $B=\oel/(\pi^m)$. Then $A$ is a truncated discrete valuation ring of length $m$ killed by $p$ and $B$ is also a truncated discrete valuation ring which is finite flat over $A$. For the non-log case, if $m=1$ then $c(L/K)<1$. This is the same as saying that $L/K$ is unramified (\cite[Proposition 6.9]{AS1}) and $c(L/K)=0$. For the log case, we have $m=e-1>1+c_{\log}(L/K)$. Hence we may assume $m\geq 2$ for both cases. In particular, $B/A$ is a finite extension of truncated discrete valuation rings. Fix a section $k\to A$ of the reduction map $A\to k$. Then we have a lift $k[[u]] \to A$ sending $u$ to the image of $\pi$. Put $F=k((u))$. Let $K'$, $K'_\infty$ and $F'$ be as in Subsection \ref{subsecimperf}.

By Lemma \ref{liftB}, we can find a finite extension $E/F$ and a cartesian diagram
\[
\xymatrix{
\oef \ar[r] \ar[d] & \oee \ar[d]\\
A\ar[r] & B. 
}
\]
Note the equality $[L:K]=[E:F]$. By Corollary \ref{LEnonlog} or Corollary \ref{LElog}, we have the equality in each of two cases
\begin{align*}
\sharp\cF^n_{F}(\oee)&=\sharp \cF^n_{K}(\oel)=[L:K]=[E:F],\\
\sharp\cF^n_{F,\log}(\oee)&=\sharp \cF^n_{K,\log}(\oel)=[L:K]=[E:F].
\end{align*}
Thus Lemma \ref{CIlift} (\ref{CIlift-sep}) or Lemma \ref{CIliftlog} implies that the extension $E/F$ is separable. Moreover, its (non-log or log) ramification is bounded by $n$.

We claim that the extension $E/F$ is abelian. Indeed, by Proposition \ref{compfunc} or Theorem \ref{complog}, we have diagrams of finite $G_{F'}$-sets
\[
\begin{array}{ll}
\xymatrix{
\cF_K(\oel)\ar[d] & \cF_F(\oee)\ar[d]\\
\cF^n_K(\oel)\ar[r]_-{\sim} &\cF^n_F(\oee),
}&
\xymatrix{
\cF_K(\oel)\ar[d] & \cF_F(\oee)\ar[d]\\
\cF^n_{K,\log}(\oel)\ar[r]_{\sim} & \cF^n_{F,\log}(\oee)
}
\end{array}
\]
whose horizontal arrow is an isomorphism in each of two cases. Since the (non-log or log) ramification is bounded by $n$, the vertical arrows are bijections compatible with the Galois action.

Since $L/K$ is Galois, the stabilizer of the $G_K$-set $\cF_K(\oel)$
\[
\{g\in G_K\mid g(\psi)=\psi\text{ for any }\psi\in \cF_K(\oel)\}
\]
is equal to $G_L$. Let $\tilde{E}$ be the Galois closure of the finite separable extension $E/F$. Then the stabilizer of the $G_{F'}$-set $\cF_F(\oee)|_{G_{F'}}$ is $G_{\tilde{E}F'}$. By the above isomorphism, it is also isomorphic to the stabilizer of the $G_{K'_\infty}$-set $\cF_K(\oel)|_{G_{K'_\infty}}$, which is equal to $G_{LK'_\infty}$. The isomorphism $G_{F'}\simeq G_{K'_\infty}$ induces an isomorphism
\[
\Gal(\tilde{E}F'/F')\simeq \Gal(LK'_\infty/K'_\infty).
\]
In particular, we have the equality $[\tilde{E}F':F']=[LK'_\infty:K'_\infty]$. Since the extension $\tilde{E}/F$ is finite separable and $F'/F$ is primary by Lemma \ref{primaryext} (\ref{primaryext-prim}), we obtain the equality $F=F'\cap \tilde{E}$. Hence
\[
[E:F]\leq [\tilde{E}:F]=[\tilde{E}F':F']=[L K'_\infty:K'_\infty]\leq [L:K]=[E:F]
\]
and $\tilde{E}$ is equal to $E$. Thus the extension $E/F$ is Galois. Moreover, we also have
\[
\Gal(E/F)\simeq \Gal(E F'/F')\simeq \Gal(L K'_\infty/K'_\infty)\simeq \Gal(L/L\cap K'_\infty)\subseteq \Gal(L/K),
\]
which implies that $E/F$ is abelian. 

By Proposition \ref{compfunc} or Theorem \ref{complog}, we also have a bijection
\[
\cF^j_K(\oel)\simeq \cF^j_F(\oee),\ \cF^j_{K,\log}(\oel)\simeq \cF^j_{F,\log}(\oee)
\]
for any positive rational number $j$ satisfying $j\leq n$. Thus we obtain the equality
\[
c(L/K)=c(E/F),\ c_{\log}(L/K)= c_{\log}(E/F).
\]
Hence the theorem follows from Theorem \ref{xiao} (\ref{xiaop}).
\end{proof}



\end{document}